\documentclass{amsart}

\usepackage{hyperref}
\usepackage{tikz-cd}
\usepackage{amsfonts,amsmath,amssymb,mathrsfs, stmaryrd, amsthm}
\usepackage{enumitem}
\usepackage{cite}
\usepackage{amsrefs}
\usepackage{bm}
\usepackage[greek,english]{babel}
\usepackage{teubner}

\title{The smallest quantum Mackey deformation}
\author{Yvann Gaudillot--Estrada}
\email{yvann.gaudillot-estrada@univ-lorraine.fr}
\thanks{This research was supported by the projects OpART (ANR-23-CE40-0016) and CroCQG (ANR-25-CE40-5010) of the \emph{Agence Nationale de la Recherche}.  It is based upon work from COST Action CaLISTA CA21109 supported by COST (European Cooperation in Science and Technology) \url{www.cost.eu}.  Part of the research was carried out during the special trimester on Representation Theory and Noncommutative Geometry at the Institut Henri Poincaré (UAR 839 CNRS-Sorbonne Université) --- LabEx CARMIN (ANR-10-LABX-59-01).}

\newtheorem{theorem}{Theorem}[section]
\newtheorem{proposition}[theorem]{Proposition}
\newtheorem{lemma}[theorem]{Lemma}
\newtheorem{corollary}[theorem]{Corollary}

\theoremstyle{definition}
\newtheorem{definition}[theorem]{Definition}
\newtheorem{remark}[theorem]{Remark}

\newtheorem*{convention}{Convention}
\newtheorem*{vocabulary}{Vocabulary}
\newtheorem{notation}[theorem]{Notation}

\setlist[itemize]{leftmargin=0.7cm}

\newcommand{\cpa}{\text{\coppa}}

\begin{document}    

\begin{abstract}
    When $G$ is a real semisimple group, there is a surprising interplay between its representation theory and that of its motion group $G_0$, known as the Mackey analogy. The present paper extends this analogy to the framework of $q$-deformations, for $G = \mathrm{SL}(2,\mathbb{R})$. In fact, we construct a deformation of  $\mathrm{SL}(2,\mathbb{R})$ parametrized by $(q,t) \in \mathbb{R}_+^* \times \mathbb{R}$, where $q$ is the quantization parameter and $t$ is the Mackey parameter. We show how the representation theory varies along this deformation and we prove an analogue of the Connes-Kasparov isomorphism for the $q$-deformed reduced group C*-algebra.
    
\end{abstract}

\maketitle

Assume that $G$ is a semisimple Lie group and $K \subset G$ is a maximal compact subgroup. The motion group of $G$ is the semidirect product $G_0 = K \ltimes (\mathfrak{g}/\mathfrak{k})$, where $\mathfrak{g}/\mathfrak{k}$ is the quotient of the Lie algebra of $G$ by the Lie algebra of $K$, viewed as an abelian Lie group and equipped with the adjoint action of $K$. On the basis of ideas of Mackey \cite{MackeyAnalogy}, Higson exhibited striking similarities between the representation theories of $G$ and $G_0$ in the complex case \cite{MAhigson}. This analogy, named after Mackey, was later carried out by Afgoustidis for real groups \cites{Afgoustidis1,Afgoustidis2,Afgoustidis3}. One of its main features is the existence of a natural bijection mapping continuously (see \cite{AfgoustidisAubert}) the tempered (or reduced unitary) dual of $G$ to the unitary dual of $G_0$ which preserves minimal $K$-types. The Mackey analogy also led to a new proof of the Connes-Kasparov isomorphism: the contraction of $G$ to $G_0$ gives rise to a continuous field of reduced group C*-algebras, called the Mackey deformation, which has constant K-theory. More recently, the authors of \cite{MEmbeddingReal} constructed an embedding $C^*(G_0) \hookrightarrow C^*_r(G)$ and used it to characterize the Mackey bijection and the Connes-Kasparov map (with a prior approach \cite{ComplexMA} in the complex case).

In a completely different direction, another interesting deformation of $G$ is given by quantization. In a recent paper \cite{DCquantisation}, De Commer constructed $q$-deformations of various convolution algebras of $G$ related to a real Poisson-Lie structure on the complexification of $G$. If $q \neq 1$ is a real positive parameter, there are $q$-analogues of the enveloping $\ast$-algebra of $\mathfrak{g}$, of the Hecke algebra associated to $(\mathfrak{g},K)$ and of the maximal group C*-algebra of $G$ which we respectively denote by $U_q(\mathfrak{g})$, $R_q(\mathfrak{g},K)$ and $C^*_q(G)$. So far, the thorough study of these algebras and their representations has only concerned complex groups and $\mathrm{SL}(2,\mathbb{R})$. For complex groups, De Commer's construction coincides with the Drinfeld double of the compact form; its structure and representation theory is now well understood \cites{VoigtYuncken, Arano2}. The case $G = \mathrm{SL}(2,\mathbb{R})$ was the object of several rather recent publications. The irreducible $\ast$-representations of $C^*_q(G)$ and the non-degenerate simple $R_q(\mathfrak{g},K)$-modules, which are the analogues of the $(\mathfrak{g},K)$-modules, are classified in \cite{DCDTsl2R} and \cite{YGE} respectively. The authors of \cite{DCDTinvariant} constructed an analogue of the regular representation leading to a natural analogue of the reduced group C*-algebra $C^*_{q,r}(G)$. The spectrum of the latter is characterized in \cite{DCinduction}.

The goal of the present paper is to unify these $q$-deformations with the Mackey deformation for $G = \mathrm{SL}(2,\mathbb{R})$, allowing for an extension of the Mackey analogy to $q$-deformed $\mathrm{SL}(2,\mathbb{R})$. Our work is inspired by what has been done by Monk and Voigt for complex semisimple groups \cite{MonkVoigt}.

Our analogue of the motion group for $q$-deformed $\mathrm{SL}(2,\mathbb{R})$ does not depend on~$q$, it is the groupoid $G_0^\star$ of the action of $K = \mathrm{SO}(2)$ on the 2-sphere $K\backslash U$, where $U = \mathrm{SU}(2)$. In what follows, we construct a two-parameter continuous field of C*-algebras $C^*_r(\mathbf{G})$ defined over $\mathbb{R}_+^* \times\mathbb{R}$, whose restriction to $\{1\} \times \mathbb{R}$ is precisely the Mackey deformation field and whose fibers, for each $q \neq 1$, are given by
$$C^*_r(\mathbf{G})_{q,t} = \begin{cases}
    C^*_{q^t,r}(G) & \text{if $t\neq 0$,}\\
    C^*(G^\star_0) &\text{if $t= 0$.}
\end{cases}$$
As we will see, this continuous field of C*-algebras is a reflection of a much more rigid field involving the corresponding (quantized) enveloping algebras.

Moreover, for every fixed $q \neq 1$, we prove that the restriction of $C^*_r(\mathbf{G})$ to $\{q\}\times[0,1]$ is isomorphic to the mapping cone field of a certain embedding $$\alpha_{0}^q : C^*(G^\star_0) \hookrightarrow C^*_{q,r}(G).$$
We also show that this embedding induces a bijection between the spectrum of $C^*(G^\star_0)$ and that of $C^*_{q,r}(G)$, as well as an isomorphism in K-theory. The latter result means that $C^*_r(\mathbf{G})_{|\{q\}\times[0,1]}$ has constant K-theory and thus can be interpreted as an analogue of the Connes-Kasparov isomorphism.

The paper is organized as follows. In the first section, we introduce or develop the operator algebraic tools needed for our purpose. In the second section, we present the various convolution algebras involved in the two-parameter deformation. The third section is devoted to the construction of the continuous field of C*-algebras $C^*_r(\mathbf{G})$. In the last section, we construct the embedding $\alpha_0^q$ and establish the above listed properties.

For the rest of the paper, we write $G = \mathrm{SL}(2,\mathbb{R})$, $K = \mathrm{SO}(2)$, $U = \mathrm{SU}(2)$, unless it is explicitly stated otherwise. We denote by $\mathfrak{g}$, $\mathfrak{k}$, $\mathfrak{u}$ the associated Lie algebras. Moreover, for any real Lie algebra $\mathfrak{a}$, the associated enveloping $\ast$-algebra $U(\mathfrak{a})$ is defined as the complexification of its universal enveloping algebra, endowed the unique $\ast$-structure such that the elements of $\mathfrak{a}$ are skew-adjoint.

\section{Operator algebraic preliminaries}

The goal of this first section is to introduce some background material that will facilitate the construction of the quantum Mackey deformation field. For the rest of the paper, the space of bounded operators and the space of compact operators over any Hilbert space $\mathcal{H}$ are denoted by $\mathfrak{B}(\mathcal{H})$ and $\mathfrak{K}(\mathcal{H})$, respectively.

\subsection{Continuous fields of Hilbert spaces and C*-algebras} Throughout the paper, several continuous fields of Banach spaces will be introduced. Here we recall some of their aspects and fix notations, referring to \cite{Dixmier}*{Chapter 10} for a complete exposition of the topic. More importantly, we introduce a notion of constraint on the continuous fields of elementary C*-algebras associated to locally trivial continuous field of Hilbert spaces. In the following, $X$ denotes any locally compact space.

Let $B$ be any continuous field of Banach spaces over $X$. The fiber of $B$ over each $x \in X$ is denoted by $B_x$ and more generally, for any locally compact subspace $Y \subset X$, the continuous field over $Y$ induced by $B$ is denoted by $B_{|Y}$. The space of continuous vector fields of $B$ is denoted by $\Gamma(B)$ and its subspace consisting of vector fields vanishing at infinity is denoted by $\Gamma_0(B)$. For any $b \in \Gamma(B)$, we write $b_x$ for the evaluation of $b$ at any $x \in X$.

Let $B'$ be another continuous field of Banach spaces over $X$. By definition, an embedding of continuous fields $T : B \to B'$ is a collection of isometries $(T_x : B_x \to B'_x)_{x \in X}$ mapping $\Gamma(B)$ into $\Gamma(B')$. Such an embedding of continuous fields $T$ is called an isomorphism if $T_x$ is surjective for all $x \in X$. In that case, \cite{Dixmier}*{§10.2.4} shows that $\Gamma(B)$ is mapped onto $\Gamma(B')$. When $B'$ is a constant field, an isomorphism $T : B \to B'$ is referred to as a trivialization of $B$

We call $B$ a continuous subfield of $B'$ if $B_x$ is a Banach subspace of $B_x'$ for all $x \in X$ and if $\Gamma(B) \subset \Gamma(B')$. We shall use repeatedly the two following facts.
\begin{itemize}
    \item Given a Banach subspace $S_x$ of $B_x$ for each $x \in X$, there is at most one continuous subfield of $B$ whose fibers are the $S_x$. Such a continuous subfield exists if and only if for any $x \in X$ and any $v \in B_x$, there exists $b \in B$ such that $b_x = v$ and $b_{x'} \in S_{x'}$ for all $x' \in X$.
    \item If $Y$ is a closed subset of $X$, then the restriction map $\Gamma_0(B) \to \Gamma_0(B_{|Y})$ is surjective.
\end{itemize}
For the first fact, the uniqueness follows from \cite{Dixmier}*{§10.2.4}. The second fact is obtained by combining \cite{Dixmier}*{§10.1.12} and \cite{Dixmier}*{§10.2.6}.

We call $B$ a continuous field of Hilbert spaces if each of its fibers is a Hilbert space. In that case, by polarization, the pointwise inner product of two continuous vector fields is continuous over $X$.

We call $B$ a continuous field of C*-algebras if each of its fibers is a C*-algebra and if $\Gamma(B)$ is stable by pointwise multiplication and involution. In that case, $\Gamma_0(B)$ has a natural C*-algebra structure. Any of its irreducible representations is of the form $\pi \circ\mathrm{ev}_x: \Gamma_0(B) \twoheadrightarrow B_x \to \mathfrak{B}(H)$, where $\mathrm{ev}_x$ denotes the evaluation at some point $x \in X$ and $(\pi, H)$ is an irreducible representation of $B_x$. In that way, the spectrum of $\Gamma_0(B)$ identifies with the disjoint union of the spectra of $B_x$ for all $x \in X$. A continuous subfield of C*-algebras of $B$ is defined as a continuous subfield whose space of continuous vector fields is a $\ast$-subalgebra of $\Gamma(B)$ (this implies that the fibers are C*-subalgebras). An embedding of continuous fields of C*-algebras is an embedding of continuous fields which is a fiberwise C*-algebra morphism.

\begin{convention}
    It is a well known fact that a continuous field of C*-algebras $B$ is uniquely characterized by the C*-algebra $\Gamma_0(B)$ and the embedding of $C_0(X)$ into its central multiplier algebra. For that reason, by abuse of notation, the C*-algebra $\Gamma_0(B)$ will also be denoted by $B$.
\end{convention}

A notion that we will use is that of pushforward of continuous field of C*-algebras. Let $Y$ be another locally compact Hausdorff space, $f : X \to Y$ an open continuous surjection and $A$ a continuous field of C*-algebras over $X$. 

\begin{proposition}\label{pushforward}
    There exists a unique continuous field of C*-algebras $f_\ast A$ over $Y$ such that
    \begin{itemize}
        \item the fiber of $f_\ast A$ over each $y \in Y$ is $A_{|f^{-1}(y)}$,
        \item if $a \in A$ then $f_\ast a = \left(a_{|f^{-1}(y)}\right)_{y \in Y}$ is a continuous vector field of $f_\ast A$.
    \end{itemize}
\end{proposition}

We call $f_\ast A$ the pushforward of $A$ by $f$. We will neither prove nor use it, but one can show that $f_* : A \to f_*A$ is an isomorphism of C*-algebras (we are ignoring the continuous field structure).
\begin{proof} The uniqueness follows from \cite{Dixmier}*{Proposition 10.2.4}, we only prove the existence.
    Let $\mathscr{S} = \{ f_\ast a : a \in A\} \subset \prod_{y \in Y} A_{|f^{-1}(y)}$. It is clear that $\mathscr{S}$ is a $\ast$-algebra for the pointwise operations and that $\mathscr{S}_y = A_{|f^{-1}(y)}$ for all $y\in Y$. Thus, if we show that $y \in Y \mapsto \Vert (f_\ast a)_y\Vert$ is continuous for every $a \in A$, then $\mathscr{S}$ will generate a continuous field of C*-algebras satisfying the conditions of the proposition.
    
    Let us fix $y_0 \in Y$ and $\epsilon > 0$. Let $x_0 \in X$ be such that $\Vert (f_\ast a)_{y_0}\Vert = \Vert a_{x_0}\Vert$ and let $K$ be a compact neighborhood of $x_0$ such that $\Vert a_x\Vert \leq \varepsilon$ for all $x \in X\setminus K$. By continuity of $x \mapsto \Vert a_x\Vert$, there exists an open set $V \subset X$ containing $f^{-1}(y_0)$ such that
    $$\forall x \in V, \exists x' \in f^{-1}(y_0), \quad \Vert a_x \Vert \leq \Vert a_{x'} \Vert + \epsilon.$$
    Let $W \subset V$ be an open neighborhood of $x_0$ such that
    $$\Vert a_{x_0}\Vert - \epsilon \leq \Vert a_x\Vert \qquad (\forall x \in W).$$
    Given that $f$ is open and $K\setminus V$ is a compact set not intersecting $f^{-1}(y_0)$, the set $U = f(W) \setminus f(K \setminus V)$ is a neighborhood of $y_0$. Then, by construction we have
    $$\Vert a_{x_0}\Vert -\epsilon \leq \Vert(f_\ast a)_y \Vert \leq \Vert a_{x_0}\Vert +\epsilon \qquad (\forall y \in U),$$
    which concludes the proof.
\end{proof}

From now, we fix a locally trivial continuous field of Hilbert spaces $\mathcal{H}$ over $X$. We denote by $\mathfrak{K}(\mathcal{H})$ the continuous field of elementary C*-algebras associated to $\mathcal{H}$, see \cite{Dixmier}*{§10.7.2}.
We recall that the fiber of $\mathfrak{K}(\mathcal{H})$ over every $x \in X$ is $\mathfrak{K}(\mathcal{H}_x)$ and that for any $\psi,\phi \in \Gamma(\mathcal{H})$, the vector field $(\eta \mapsto \langle \psi_x, \eta \rangle \phi_x)_{x \in X}$ of $\mathfrak{K}(\mathcal{H})$ is continuous.

As we will see, every deformed group C*-algebra involved in this paper can be embedded into some $\mathfrak{K}(\mathcal{H})$ for a well-chosen locally trivial continuous field of Hilbert spaces $\mathcal{H}$. In order to be able to describe them precisely, we introduce the notion of constraint.

\begin{definition}
    A preconstraint $\mathscr{C}$ on $\mathcal{H}$ is a collection $(\mathcal{H}_{x,j})_{j \in J_x}$, for each $x \in X$, of mutually orthogonal non-zero closed subspaces of $\mathcal{H}_x$ whose direct sum is dense in $\mathcal{H}_x$. We write $\mathscr{C} = (\mathcal{H}_{x,j})_{x \in X,j \in J_x}$.

    A local fixation of a such a preconstraint $\mathscr{C}$ is a pair $(u,V)$, where $V$ is an open subset of $X$ and $u$ is a trivialization of $\mathcal{H}_{|V}$ such that for every $x \in V$, there exists an open subset $V_x \subset V$ satisfying:
    \begin{equation}\label{semicontinuity}\forall y \in V_x,\; \forall j' \in J_y, \;\exists j \in J_x, \quad u_x(\mathcal{H}_{x,j}) \subset u_y(\mathcal{H}_{y,j'}).\end{equation}
    We call $\mathscr{C}$ a constraint if it admits a collection of local fixations $\{(u_\alpha,V_\alpha)\}_\alpha$ such that $X = \cup_\alpha V_\alpha$
\end{definition}

\begin{proposition}
    Let $\mathscr{C} = (\mathcal{H}_{x,j})_{x \in X,j \in J_x}$ be a constraint on $\mathcal{H}$. There is a unique continuous subfield of $\mathfrak{K}(\mathcal{H})$ whose fiber at each $x \in X$ is $\bigoplus_{j \in J_x} \mathfrak{K}(\mathcal{H}_{x,j})$.
\end{proposition}

\begin{proof}
    Let us write $\mathfrak{K}_\mathscr{C}(\mathcal{H})_x=\bigoplus_{j \in J_x} \mathfrak{K}(\mathcal{H}_{x,j})$ for each $x \in X$ and let $\Gamma_\mathscr{C}$ be the space $\Gamma[\mathfrak{K}(\mathcal{H})] \cap \prod_{x \in X}\mathfrak{K}_\mathscr{C}(\mathcal{H})_x$. In order to prove the proposition, it is enough to check that for every $x \in X$, any element of $\mathfrak{K}_\mathscr{C}(\mathcal{H})_x$ can be obtained as the evaluation at $x$ of an element of $\Gamma_\mathscr{C}$. Let $x \in X$ and $a \in \mathfrak{K}_\mathscr{C}(\mathcal{H})_{x}$. By considering a local fixation of $\mathscr{C}$, one can find an open neighborhood $V$ of $x$ and a continuous vector field $b$ of $\mathfrak{K}(\mathcal{H})_{|V}$ such that $b_{x} = a$ and $b_{y} \in \mathfrak{K}_\mathscr{C}(\mathcal{H})_y$ for all $y \in V$. Let $\varphi$ be a continuous function on $V$ with compact support such that $\varphi(x) = 1$. Extending $\varphi b$ by zero outside $V$, we get an element of $ \Gamma_\mathscr{C}$ whose evaluation at $x$ is $a$.
\end{proof}

The continuous field of C*-algebras defined by the above proposition will be denoted by $\mathfrak{K}_\mathscr{C}(\mathcal{H})$. Considering the latter as a C*-algebra, one can easily determine its irreducible representations.

\begin{lemma}\label{irrepsK} The collection of all the representations $$\mathrm{pr}_j \circ \mathrm{ev}_x: \mathfrak{K}_\mathscr{C}(\mathcal{H}) \to \bigoplus_{l \in J_x} \mathfrak{K}(\mathcal{H}_{x,l}) \to \mathfrak{K}(\mathcal{H}_{x,j}) \qquad (x \in X, j\in J_x),$$
where $\mathrm{ev}_x$ denotes the evaluation at $x$ and $\mathrm{pr}_j$ the projection on the component indexed by $j$, is a complete set of representatives of the spectrum of $\mathfrak{K}_\mathscr{C}(\mathcal{H})$.
\end{lemma}

\begin{proof}
    The spectrum of $\mathfrak{K}_\mathscr{C}(\mathcal{H})$ is the disjoint union of the spectra of its fibers. For every $x \in X$, the irreducible representations of $\bigoplus_{j \in J_x} \mathfrak{K}(\mathcal{H}_{x,j})$, the fiber of $\mathfrak{K}_\mathscr{C}(\mathcal{H})$ over $x$, are up to equivalence its actions on the spaces $\mathcal{H}_{x,j}$, as $j$ ranges over $J_x$.
\end{proof}

\subsection{The generalized Stone-Weierstrass theorem}
The main tool that we use to determine the explicit form of the C*-algebras involved in our two-parameter deformation is a generalization of the Stone-Weierstrass theorem for postliminal \cite{Dixmier}*{§4.3.1} C*-algebras. Since we will invoke it several times, we state it here for convenience.

\begin{theorem} \label{StoneW}
    Let $\varphi : A \to B$ be a morphism of C*-algebras, with $B$ postliminal. Assume that
    \begin{enumerate}[label = (\roman*)]
        \item for every irreducible representation $\pi$ of $B$, the pullback representation $\varphi^*\pi$ of $A$ is irreducible,
        \item the induced map $\varphi^* : \hat{B} \to \hat {A}$ between the spectra of both algebras is bijective.
    \end{enumerate}
    Then $\varphi$ is an isomorphism.
\end{theorem}

This version is a combination of \cite{Dixmier}*{§2.7.3} and \cite{Dixmier}*{§11.1.6}. In the following, we actually use this theorem only for liminal C*-algebras \cite{Dixmier}*{§4.2.1}. By definition, the latter are the C*-algebras which act on their irreducible representations by compact operators. In particular, the continuous field of C*-algebras $\mathfrak{K}_\mathscr{C}(\mathcal{H})$ above are liminal.

As an application of the above theorem, let us describe the Jacobson topology of the spectrum of $\mathfrak{K}_\mathscr{C}(\mathcal{H})$, for $\mathcal{H}$ a locally trivial continuous field of Hilbert spaces over a locally compact space $X$ and a constraint $\mathscr{C} = (\mathcal{H}_{x,j})_{x \in X, j \in J_x}$ on it. Let us write $J = \sqcup_{x \in X}J_x$. For simplicity, for any $x \in X$ and $j \in J_x$, we write $\mathcal{H}_j = \mathcal{H}_{x,j}$ and we denote by $\pi_j$ the irreducible representation of $\mathfrak{K}_\mathscr{C}(\mathcal{H})$ on $\mathcal{H}_j$ from Lemma \ref{irrepsK}. For any local fixation $(u,V)$ of $\mathscr{C}$, every $x \in V$ and $j \in J_x$, we define
$$J(u,V)_{x,j} = \bigsqcup_{y\in V} \{j' \in J_y : u_x(\mathcal{H}_j)\subset u_y(\mathcal{H}_{j'})\}.$$
\begin{lemma}
    There is a unique topology on $J$ such that for every $x \in X$ and $j \in J_x$, the family of all the subsets of the form $J(u,V)_{x,j}$, for $(u,V)$ a local fixation of $\mathscr{C}$ such that $x \in V$, is a base of neighborhood of $j$.
\end{lemma}
\begin{proof}
    Let us fix $x \in X$ and $j\in J_x$. It is enough to check that (A) if $(u,V)$ and $(u',V')$ are local fixations such that $x \in V\cap V'$, then there exists an open subset $W \subset V\cap V'$ cotaining $x$ such that $J(u'_{|W},W)_{x,j}\subset J(u,V)_{x,j}$ ca and (B) for every local fixation $(u,V)$ such that $x \in V$, there exists an open subset $W \subset V$ containing $x$ such that for every $y \in W$ and $j'\in J(u,V)_{x,j}$, we have $J(u_{|W},W)_{y,j'} \subset J(u,V)_{x,j}$. The claim (B) is a straightforward consequence of the definition of local fixations. Let us prove the assertion (A). Without loss of generality, we can assume that $V = V'$, that $\mathcal{H}_{|V}$ is constant and that $u'$ is the identity. This way $(u_x)_{y \in V}$ is just a strongly continuous family of unitaries. Let $V_x$ be an open subset of $V$ satisfying condition (\ref{semicontinuity}) for both $u$ and $u'$. If $\xi$ is a non-zero element $\mathcal{H}_j$ then the functions $N_\xi :y \mapsto \langle u_x\xi, u_y\xi\rangle$ is continuous, so there exists an open subset $W \subset V_x$ on which $N_\xi$ does not vanish at all. Next, for every $y \in W$ and $j' \in J_y$ such that $\mathcal{H}_j \subset \mathcal{H}_{j'}$, the spaces $u_x(\mathcal{H}_j)$ and $u_y(\mathcal{H}_{j'})$ are not orthogonal, hence the first one is contained in the second one. We thus have $J(u'_{|W}, W)_{x,j} \subset J(u,V)_{x,j}$.
\end{proof}
Let us endow $J$ with the topology defined by the above lemma.
\begin{proposition} For any closed ideal $I$ of $\mathfrak{K}_\mathscr{C}(\mathcal{H})$, we define
$$Z(I) = \{j \in J : \forall \sigma \in I, \,\pi_j(\sigma) = 0\}.$$
If $F$ is a closed subset of $J$, let us write $$\mathscr{I}(F) = \{\sigma \in \mathfrak{K}_\mathscr{C}(\mathcal{H}) : \forall j \in F,\, \pi_j(\sigma) = 0 \}.$$ Then $I \mapsto Z(I)$ establishes a one-to-one correspondence between closed ideals of $\mathfrak{K}_\mathscr{C}(\mathcal{H})$ and closed subsets of $J$. The inverse is given by $F \mapsto \mathscr{I}(F)$.
\end{proposition}

\begin{proof}
    Let us check first that if $I$ is a closed ideal of $\mathfrak{K}_\mathscr{C}(\mathcal{H})$, then $Z(I)$ is closed in~$J$. Let $x \in X$ and $j \in J_x \setminus Z(I)$. Let $\sigma \in I$ such that $\pi_j(\sigma) \neq 0$. By continuity of~$\sigma$, there exists a local fixation $(u,V \ni x)$ such that for any $y \in V$, the restriction of $u_y\sigma_y u_y^*$ to $u_x(\mathcal{H}_j)$ is non-zero. Therefore, for any $j' \in J(u,V)_{x,j}$, we have $\pi_{j'}(\sigma) \neq 0$. So $Z(I)$ is closed.

    Now, let us prove that $\mathscr{I}(Z(I)) = I$. A complete set of representatives of the spectrum of $I$ is given by $\{\pi_j : j \in J \setminus Z(I)\}$. Since every irreducible representation of $I$ has a unique extension into an irreducible representation of $\mathscr{I}(Z(I))$ and the latter vanishes on $Z(I)$, the set $\{\pi_j : j \in J \setminus Z(I)\}$ is also a complete set of representatives for the spectrum of $\mathscr{I}(Z(I))$. Because $\mathfrak{K}_\mathscr{C}(\mathcal{H})$ is liminal, so is $\mathscr{I}(Z(I))$. Then, by applying the generalized Stone-Weierstrass theorem, we get that $\mathscr{I}(Z(I)) = I$.

    Finally, we show that $Z(\mathscr{I}(F)) = F$ for any closed $F \subset J$. Let $x \in X$ and $j \in J_x \setminus F$. Let us consider a neighborhood $J(u,V)_{x,j}$ of $j$ which is contained in $J\setminus F$. By restricting $V$ if necessary, we can assume that $V = V_x$ in the notation of (\ref{semicontinuity}). Then, there exists a continuous vector field $\sigma$ of $\mathfrak{K}_\mathscr{C}(\mathcal{H})$ vanishing on $X\setminus V$, such that $\pi_{j}(\sigma) \neq0$ and
    $$(u_y\sigma_yu_y^*)_{|u_x(\mathcal{H}_{j})^\perp} = 0 \qquad (y \in V).$$
    We have $\pi_{j'}(\sigma) = 0$ for all $j' \in J \setminus J(u,V)_{x,j}$, therefore $\sigma \in \mathscr{I}(F)$. This shows that $Z(\mathscr{I}(F)) \subset F$. The converse inclusion is straightforward.
\end{proof}

\begin{corollary} \label{Jacobson}
    The map $j \mapsto \pi_j$ induces a homeomorphism from $J$ onto the spectrum of $\mathfrak{K}_\mathscr{C}(\mathcal{H})$, the latter being equipped with the Jacobson topology.
\end{corollary}

\subsection{Representation theory of crossed products} The full C*-algebras of the motion group $G_0$ and its groupoid analogue $G_0^\star$ (see the introduction) are examples of crossed products of a commutative C*-algebra by a compact group. We conclude this section by a well-known description of the spectrum of such algebras. For a complete exposition on crossed product, see \cite{Dana} for instance.

We fix a commutative C*-algebra $A$, a compact group $K$ and an action $\alpha$ of $K$ on $A$ by automorphisms. Let $K \ltimes A$ be the associated crossed product C*-algebra. We already stress that the latter coincides with the C*-algebra of the transformation groupoid $\mathrm{Sp}A \curvearrowleft K$ consisting of the right action of $K$ induced by $\alpha$ on the spectrum of $A$, see \cite{Renault}*{p.59}.

Consider a point $x \in \mathrm{Sp}A$ and $\pi : \Gamma \to \mathfrak{B}(\mathcal{H})$ a unitary representation of a closed subgroup $\Gamma$ of the stabilizer $K_x$ of $x$. Let us write
$$L(K,\mathcal{H})^{\Gamma} = \{\psi \in L^2(K,\mathcal{H}) : \forall (\gamma,k) \in \Gamma \times K, \; \psi(\gamma k) = \pi(\gamma)\psi(k) \}.$$
We endow $L(K,\mathcal{H})^{\Gamma}$ with the right regular action of $K$ and the action of $A$ defined as follows:
$$\forall (a,\psi) \in A \times L(K,\mathcal{H})^{\Gamma}, \forall k \in K, \quad (a \psi)(k) = a(xk)\psi(k).$$
These actions form a covariant representation of the dynamical system $(A, K,\alpha)$ and thus induce a representation of $K \ltimes A$ on $L(K,\mathcal{H})^{\Gamma}$ that we denote by $\mathrm{Ind}(x,\pi)$, or $\mathrm{Ind}(x,\Gamma , \pi)$ when we want $\Gamma$ to be explicit.

\begin{remark}\label{InductionStages}
    Let us consider the same situation as above and $\Gamma'$ another closed subgroup of $K$ fixing $x$. Assuming that $\Gamma'$ contains $\Gamma$, induction in stages implies that $\mathrm{Ind}(x,\Gamma , \pi)$ and $\mathrm{Ind}(x, \Gamma', \mathrm{Ind}_\Gamma^{\Gamma'}(\pi))$ are canonically equivalent.
\end{remark}

\begin{theorem}\label{CrossedProduct} The following holds.

\begin{enumerate}[label = (\roman*)] 
    \item For every $x \in \mathrm{Sp}A$ and every irreducible unitary representation $\pi$ of $K_x$, the representation $\mathrm{Ind}(x,\pi)$ of $K \ltimes A$ is irreducible.
    \item Two such representations $\mathrm{Ind}(x,\pi)$ and $\mathrm{Ind}(x',\pi')$ are equivalent if and only if there exists $k \in K$ such that $x'=xk$ and $\pi$ is equivalent to $\pi'(k^{-1}\cdot k)$.
    \item Any irreducible representation of $A$ is equivalent to one of the form $\mathrm{Ind}(x,\pi)$ for $x \in \mathrm{Sp}A$ and $\pi$ an irreducible unitary representation of $K_x$.
\end{enumerate}
\end{theorem}

A proof of these statements can be found in \cite{MonkVoigt}*{§4.1}.

\section{Deformed reduced C*-algebras of $\mathrm{SL}(2,\mathbb{R})$ and their duals}

Our goal now is to introduce the convolution algebras involved in our two-parameter deformation of $\mathrm{SL}(2,\mathbb{R})$ and recall their representation theory. In the next section, we will construct a continuous field of C*-algebras $C^*_r(\mathbf{G})$ defined over $\mathbb{R}^*_+ \times \mathbb{R}$ whose fibers will be, depending on the base point, the reduced C*-algebra of $G$, the group C*-algebra of its Cartan motion group $G_0$, one of the $q$-deformed reduced C*-algebras $C^*_{q,r}(G)$ mentioned in the introduction, or the groupoid C*-algebra $C^*(G^\star_0)$. Here, we describe explicitly the spectrum of these C*-algebras. In order to characterize the corresponding irreducible representations, it will be convenient to express how certain finitely generated algebras of unbounded multipliers act on them. For $C^*_r(G)$ and $C^*(G_0)$, these algebras will be the enveloping $\ast$-algebras of the Lie algebras of $G$ and $G_0$ respectively.

\begin{convention}
    Let $A$ be any $\ast$-algebra. An inner product $(\cdot |\cdot)$ on an $A$-module $V$ is said to be $A$-invariant if $(w |av) = (a^*w | v)$ for all $v,w \in V$ and $a \in A$. By a unitary $A$-module, we mean an $A$-module endowed with a fixed $A$-invariant inner product. We call an $A$-module unitarizable if it admits an $A$-invariant inner product.

We will use specific notations for certain subsets of $\mathbb{Z}$. First, the set of even and odd integer are denoted by $\mathbb{Z}^\mathrm{even}$ and $\mathbb{Z}^\mathrm{odd}$ respectively. If $\varepsilon \in \{-1,1\}$ then we write
$\mathbb{Z}^\varepsilon$ for set of integers having the same parity as $(\varepsilon -1)/2$. For any $\lambda \in \mathbb{Z}$, the notation 
$\mathbb{Z}_{>\lambda}$ (respectively $\mathbb{Z}_{\geq\lambda}$, $\mathbb{Z}_{<\lambda}$, $\mathbb{Z}_{\leq\lambda}$) refers to the set of integers $n$ such that $n >\lambda$ (respectively $n \geq \lambda$, $n <\lambda$, $n \leq\lambda$). We will often combine these notations, so that, for example, $\mathbb{Z}^\varepsilon_{>\lambda}$ means $\mathbb{Z}^\varepsilon\cap \mathbb{Z}_{>\lambda}$.

The space of regular functions on $K$ (the $K$-finite functions) is denoted by $\mathcal{O}(K)$. We define the following element of $i\mathfrak{k}$:
\begin{equation}\label{thetadef} \theta = \begin{pmatrix}
    0 & i \\ -i & 0
\end{pmatrix}.\end{equation}
The family $(\zeta_n)_{n \in \mathbb{Z}}$ of characters of $K$ defined by
$$\zeta_n(e^{it\theta}) = e^{int} \qquad (t \in \mathbb{R}),$$
is a basis of $\mathcal{O}(K)$. For each $\varepsilon\in \{-1,1\}$, we denote by $\mathcal{O}^\varepsilon(K)$ the space of regular functions having the same parity as $(\varepsilon -1)/2$, that is, $\mathrm{span}(\zeta_n : n \in \mathbb{Z}^\varepsilon)$.
\end{convention}

\subsection{Representation theory of $\mathrm{SL}(2,\mathbb{R})$} We begin by recalling some aspects of the representation theory of $G$. Since this is discussed in many textbooks, for example \cites{Vogan, Wallach}, we will essentially focus on fixing notation for future use.

The Hecke algebra of the pair $(\mathfrak{g},K)$, defined in \cite{KnappVoganInd}*{§I.3}, is denoted by $R(\mathfrak{g},K)$. It is the convolution algebra of $K$-finite distributions on $G$ supported in $K$. As such, it is naturally endowed with a $\ast$-structure. The category of non-degenerate (or approximately unital) $R(\mathfrak{g},K)$-modules corresponds to the category of $(\mathfrak{g},K)$-modules, see \cite{KnappVoganInd}*{Theorem 1.117}. Accordingly, a $(\mathfrak{g},K)$-module is said to be unitary or unitarizable if the associated $R(\mathfrak{g},K)$-module is.

The unitary representation theory of $G$ essentially reduces to the study of the $(\mathfrak{g},K)$-modules. As an illustration of this principle, there is a one-one correspondence between irreducible unitary representations of $G$ and unitarizable simple $(\mathfrak{g},K)$-modules. It is defined as follows: if $V$ is a simple unitary $(\mathfrak{g},K)$-module then its Hilbert completion $\bar{V}$ is an irreducible representation of $G$ and $V$ corresponds to the space $K$-finite vectors of $\bar{V}$, see \cite{Baldoni}*{§3}.

The main way to produce interesting representations of $G$ is parabolic induction. Let $(E,F,H)$ be the canonical basis of $\mathfrak{g}$. Consider the minimal parabolic subgroup $P$ of $G$ with Lie algebra $\mathfrak{p} = \mathrm{span}(H,E)$. Let $R(\mathfrak{p}, M)$ be the Hecke algebra relative to the pair $(\mathfrak{p},M)$, where $M = K\cap P = \{-1,1\}$. If $V$ is a $(\mathfrak{p},M)$-module then let us denote by $I_{\mathfrak{p},M}^{\mathfrak{g},K}(V)$ the $(\mathfrak{g},K)$-module given by the $K$-finite part of $\mathrm{Hom}_{\mathfrak{p},M}(R(\mathfrak{g},K), V)$, see \cite{KnappVoganInd}*{§II.1}. The one-dimensional $(\mathfrak{p},M)$-modules are given by the family of characters $(\chi_{\varepsilon,\lambda})_{\varepsilon = \pm 1, \lambda \in \mathbb{C}}$ of $R(\mathfrak{p},M)$ characterized by
$$\chi_{\varepsilon,\lambda}(-1 \in M) = \varepsilon, \qquad \chi_{\varepsilon,\lambda}(H) = \lambda.$$ If $V$ is a $(\mathfrak{p},M)$-module then $I_{\mathfrak{p},M}^{\mathfrak{g},K}(V\otimes \chi_{1,1})$ is called the $(\mathfrak{g},K)$-module parabolically induced from $V$, we simply denote it by $\mathrm{ind}(V)$. When $V$ is the $M$-finite part of an admissible unitary representation $\bar{V}$ of $L$, $\mathrm{ind}(V)$ identifies with the $K$-finite part of the induced representation $\mathrm{Ind}_P^G(\bar{V})$.

For every $(\varepsilon, \lambda) \in \{-1,1\}\times \mathbb{C}$, let us write $\mathrm{ind}(\varepsilon,\lambda)$ as a short hand for $\mathrm{ind}(\chi_{\varepsilon,\lambda})$. In compact picture, $\mathrm{ind}(\varepsilon,\lambda)$ identifies with $\mathcal{O}^\varepsilon(K)$, equipped with the right regular action of $K$. The action of $\mathfrak{g}$ on $\mathrm{ind}(\varepsilon,\lambda)$ is given by
\begin{gather}\label{actioninducedclassical}
\begin{gathered}
        \theta \zeta_n = n \zeta_n, \\(H\mp i(E+F))\zeta_n = (\lambda +1 \pm n) \zeta_n,\end{gathered}
\end{gather}
where one can note that $\theta = i(E-F)$. The $(\mathfrak{g},K)$-module $\mathrm{ind}(\varepsilon, \lambda)$ is simple unless $\lambda \in\mathbb{Z}^{-\varepsilon}$. If $\lambda \in \mathbb{Z}_{\geq 0}^{-\varepsilon}$, then $\mathrm{ind}(\varepsilon, \lambda)$ has exactly two simple submodules, that are $D^+(\lambda) = \mathrm{span}(\zeta_n :  n \in \mathbb{Z}^\varepsilon_{>\lambda})$ and $D^-(\lambda) = \mathrm{span}(\zeta_n :  n \in \mathbb{Z}^\varepsilon_{<-\lambda})$.

Let us now describe the spectrum of $C^*_r(G)$, which coincides with the tempered dual of $G$ (this a general fact concerning semisimple Lie groups \cite{L2matrixcoef}).

\begin{proposition}
    The tempered simple $(\mathfrak{g},K)$-modules consist of:
\begin{itemize}
    \item the even unitary principal series $\mathrm{ind}(1,\lambda)$, $\lambda \in i\mathbb{R}_+$,
    \item the irreducible odd principal series $\mathrm{ind}(-1,\lambda)$, $\lambda \in i\mathbb{R}_+\setminus\{0\}$,
    \item the discrete series and their limits, which are the simple submodules of $\mathrm{ind}(\varepsilon, n)$ for $\varepsilon \in \{-1,1\}$ and $n \in \mathbb{Z}_{\geq 0}^\varepsilon$.
\end{itemize}
Each of these modules has, up to a scalar, a unique invariant inner product. The corresponding completions form a complete set of representatives of the tempered dual of $G$.
\end{proposition}

\subsection{The unitary dual of the motion group}
Let us continue this introductory section with the explicit description of the unitary irreducible representations of the motion group $G_0$. We recall that the latter is the semidirect product $K \ltimes (\mathfrak{g}/\mathfrak{k})$, where $\mathfrak{g}/\mathfrak{k}$ is considered as an abelian Lie group and is equipped with the adjoint action of $K$. The full C*-algebra of $G_0$, which coincides with its reduced C*-algebra because $K$ is amenable, identifies with the crossed product $K \ltimes C^*(\mathfrak{g}/\mathfrak{k})$. The classification of the unitary irreducible representations of $G_0$ hence derives from Theorem \ref{CrossedProduct}. In order to make it explicit, we just determine the orbits and stabilizers of the spectrum of $C^*(\mathfrak{g}/\mathfrak{k})$ under the action of $K$.

We denote by $\mathfrak{g}_0$ the Lie algebra of $G_0$. Let us consider $\mathfrak{g}/\mathfrak{k}$ as its own Lie algebra, so that $\mathfrak{g}/\mathfrak{k} \subset \mathfrak{g}_0$. We introduce the following elements of $U(\mathfrak{g}/\mathfrak{k})$:
$$\delta X = H/2 \, \mathrm{mod} \,\mathfrak{k}, \qquad \delta Z = i(E \,\mathrm{mod}\, \mathfrak{k}).$$
The reasons for this notation will become clear later. Let us consider $K$ as embedded into $G_0$, so that $\theta \in i\mathfrak{k} \subset U(\mathfrak{g}_0)$. Then one can realize $U(\mathfrak{g}_0)$ as the universal unital $\ast$-algebra generated by $\theta$, $\delta X$, $\delta Z$ and subject to the relations
\begin{equation}\begin{gathered}
\theta^* = \theta, \quad \delta X^*=-\delta X, \quad \delta Z^* = \delta Z,\\
\delta X \delta Z = \delta Z \delta X, \quad \theta \delta X - \delta X\theta = -2\delta Z, \quad \theta\delta Z - \delta Z\theta = -\delta X.
\end{gathered}\end{equation}

Let us identify the spectrum of $C^*(\mathfrak{g}/\mathfrak{k})$ with the Pontryagin dual $(\mathfrak{g}/\mathfrak{k})\,\hat{}$ of $\mathfrak{g}/\mathfrak{k}$. By differentiation, the latter in turn identifies with the set of characters of the $\ast$-algebra $U(\mathfrak{g}/\mathfrak{k})$, which is in one-to-one correspondence with $\mathbb{C}$ as follows: each $\lambda \in \mathbb{C}$ corresponds to the unique character of $U(\mathfrak{g}/\mathfrak{k})$ mapping $\delta X + \delta Z$ to $\lambda$. The right action of $K$ on $(\mathfrak{g}/\mathfrak{k})\,\hat{}$, which derives from the adjoint action on $\mathfrak{g}/\mathfrak{k}$, is then given by
$$\lambda \cdot e^{it\theta} = e^{-2it}\lambda \qquad (t\in \mathbb{R}, \lambda \in \mathbb{C} = (\mathfrak{g}/\mathfrak{k})\,\hat{}\,).$$
For each $\lambda \in i\mathbb{R}_+$, let $p_\lambda$ denote the point of $\lambda /2 \in\mathrm{Sp(\mathfrak{g}/\mathfrak{k})}$. Each orbit of $(\mathfrak{g}/\mathfrak{k})\,\hat{}$ under the action of $K$ is of the form $p_\lambda K$ for a unique $\lambda \in i\mathbb{R}_+$. Moreover, the stabilizer of $p_\lambda$ is $M = \{-1,1\}$ if $\lambda \neq 0$, otherwise it is $K$. As an abelian group, $M$ is self-dual. On the other hand, the Pontryagin dual of $K$ identifies with $\mathbb{Z}$ via the pairing $(k,n) \in K \times \mathbb{Z} \mapsto \zeta_n(k)$. Applying Theorem \ref{CrossedProduct}, we get the following.

\begin{proposition} A complete set of representatives of the unitary dual of $G_0$ is
\begin{itemize}
    \item $\mathrm{Ind}(p_\lambda, M, \varepsilon)$ for $\lambda \in i\mathbb{R}_+^*$ and $\varepsilon \in \hat M=\{-1,1\}$,
    \item $\mathrm{Ind}(p_0, K, n)$ for $n \in \hat{K} = \mathbb{Z}$.
\end{itemize}
\end{proposition}
Note that representations of the second kind are one-dimensional. Besides, according to Remark \ref{InductionStages}, we have $\mathrm{Ind}(p_0,M,\varepsilon) = \bigoplus_{n \in \mathbb{Z}^\varepsilon} \mathrm{Ind}(p_0, K, n)$ for each $\varepsilon \in \{-1,1\}$, where the underlying space of $\mathrm{Ind}(p_0, K, n)$ is simply $\mathbb{C}\zeta_n$.

\begin{notation}
    For any point $x \in \mathrm{Sp(\mathfrak{g}/\mathfrak{k})}$, any subgroup $\Gamma \subset K_x$ and any irreducible unitary representation $\pi$ of $\Gamma$, let us denote by $\mathrm{ind}(x,\Gamma,\pi)$ the unitary $U(\mathfrak{g}_0)$-module consisting in the space of $K$-finite vectors of $\mathrm{Ind}(x, \Gamma, \pi)$. For more information, see \cite{Warner}*{§4.5.5}.
\end{notation}

For every $\lambda \in i\mathbb{R}_+$ and $\varepsilon\in  \{-1,1\}$, the underlying space of $\mathrm{Ind}(p_\lambda,M,\varepsilon)$ is
$$L^2(K)^\varepsilon = \{ \psi \in L^2(K) : \forall k \in K, \,\psi(-k) = \varepsilon\psi(k)\},$$
so $\mathrm{ind}(p_\lambda,M,\varepsilon)$ is equal to $\mathcal{O}^\varepsilon(K)$ as a space. The explicit action of $U(\mathfrak{g}_0)$ on it is characterized by
\begin{gather}\label{actionG0} \begin{gathered}\delta T^\pm \zeta_n = \lambda\zeta_{n\pm 2} \\\theta \zeta_n = n \zeta_n\end{gathered} \quad\qquad (n \in \mathbb{Z}^\varepsilon),\end{gather}
where $\delta T^{\pm} = 2(\delta X \mp \delta Z) = \delta X - \delta X^* \mp 2\delta Z$.

\subsection{The quantization of $\mathrm{SL}(2,\mathbb{R})$}

The $q$-deformation of $G$ we are interested in was first introduced by De Commer and Dzokou Talla \cite{DCDTsl2R}. Its construction is based on the quantization of $SL(2,\mathbb{C})$, the latter considered as the classical Drinfeld double of the Poisson Lie group $\mathrm{SU}(2)$, and is inspired by Letzter's theory of quantum symmetric pairs \cite{Letzter99}. Let us introduce the associated convolution algebras and recall some aspects of their representation theory. We shall be brief, referring to \cites{DCDTsl2R, DCquantisation, YGE} for more details.

As in our previous work \cite{YGE}, we consider algebras defined over involutive extensions of $\mathbb{C}$. This generality will be used in the next section to define an algebraic version of the quantum Mackey deformation field. In what follows, $\mathbb{F}$ denotes a field extension of $\mathbb{C}$ which is equipped with a $\ast$-algebra structure. An algebra $A$ defined over $\mathbb{F}$ having a $\ast$-algebra structure is called an $(\mathbb{F},\ast)$-algebra if we have $(\lambda a)^* = \lambda^*a^*$ for every $(\lambda,a) \in \mathbb{F}\times A$. If moreover $A$ has a structure of Hopf algebra (over $\mathbb{F}$), it is said to be a Hopf $(\mathbb{F}, \ast)$-algebra if the coproduct and the counit are morphisms of $\ast$-algebras. Dual pairings or skew-pairings of Hopf $(\mathbb{F}, \ast)$-algebras are said to be unitary if they are compatible with the $\ast$-structure, see \citelist{\cite{KlimykSchmudgen}*{§1.2.7} \cite{VoigtYuncken}*{§4.2.4}}. The sumless Sweedler notation will be used freely in the following.

We fix a non-zero element $q \in \mathbb{F}$ which admits a self-adjoint square root $q^{1/2}$ and is not a root of unity. For any integer $n$, we write $[n]_q$ for $(q - q^{-1})^{-1}(q^n-q^{-n})$. Such an element of $\mathbb{F}$ is called a $q$-integer.

Consider the Drinfeld-Jimbo algebra $U_q(\mathfrak{sl_2})$ defined over $\mathbb{F}$, with deformation parameter $q$ \cite{Jantzen}*{Chapter 1}. If $E,F,k$ denote its canonical generators, there is a unique Hopf $(\mathbb{F},\ast)$-algebra structure on $U_q(\mathfrak{sl_2})$ such that $k$ is a self-adjoint grouplike element and
$$E^* = Fk, \qquad \Delta(E) = E \otimes 1 + k \otimes E.$$
With this $\ast$-structure, we denote it by ${}_\mathbb{F}U_{q}(\mathfrak{u})$. This quantized enveloping $\ast$-algebra of $\mathfrak{u}$ is dually paired with the quantized coordinate algebra ${}_\mathbb{F}\mathcal{O}_q(U)$ of $U$. The latter is defined as the unital $(\mathbb{F},\ast)$-algebra freely generated by the entries of a $2\times 2$ matrix $u$ such that
$$u^*u=uu^* = 1, \qquad u^* = \mathrm{diag}(1,q)\mathrm{adj}(u)\mathrm{diag}(1,q^{-1}),$$
where $\mathrm{adj}(u)$ denotes the adjugate matrix of $u$. We endow it with the unique Hopf $(\mathbb{F},\ast)$-algebra structure such that $\Delta(u)$ is equal to the square of $u$ in the algebra of 2$\times$2 matrices with entries in the tensor algebra of ${}_\mathbb{F}U_{q}(\mathfrak{u})$. One can show that there is a unique unitary dual pairing of Hopf $(\mathbb{F},\ast)$-algebra between ${}_\mathbb{F}U_{q}(\mathfrak{u})$ and ${}_\mathbb{F}\mathcal{O}_q(U)$ such that
$$\langle E, u\rangle  = \begin{pmatrix}
    0 & q^{1/2}\\
    0 & 0
\end{pmatrix}, \qquad \langle k, u \rangle = \begin{pmatrix}
    q & 0 \\ 0 & q^{-1}
\end{pmatrix}.$$
This bilinear map can also be seen as a unitary skew-pairing
$${}_\mathbb{F}U_{q}(\mathfrak{u})^{\mathrm{cop}} \otimes {}_\mathbb{F}\mathcal{O}_q(U) \longrightarrow \mathbb{F},$$
where ${}_\mathbb{F}U_{q}(\mathfrak{u})^{\mathrm{cop}}$ is the Hopf $(\mathbb{F},\ast)$-algebra with the same structure as ${}_\mathbb{F}U_{q}(\mathfrak{u})$ except that it is equipped with the opposite coproduct. We denote by ${}_\mathbb{F}U_{q}(\mathfrak{g}_\mathbb{C})$ the associated quantum double \citelist{\cite{KlimykSchmudgen}*{§8.2.1} \cite{VoigtYuncken}*{§2.2.4}}. Roughly, it is an Hopf $(\mathbb{F},\ast)$-algebra containing both ${}_\mathbb{F}\mathcal{O}_q(U)$ and $ {}_\mathbb{F}U_{q}(\mathfrak{u})^{\mathrm{cop}}$ as Hopf $(\mathbb{F},\ast)$-subalgebras such that
$$yx = \langle x_{(1)}, y_{(1)}\rangle x_{(2)}y_{(2)} \langle x_{(3)}, S^{-1}(y_{(3)})\rangle \qquad (x \in {}_\mathbb{F}U_{q}(\mathfrak{u})^{\mathrm{cop}}, y \in {}_\mathbb{F}\mathcal{O}_q(U)),$$
and which is universal for these properties.

Next, we denote by ${}_\mathbb{F}U_{q}(\mathfrak{k})$ the unital subalgebra of ${}_\mathbb{F}U_{q}(\mathfrak{u})$ generated by the element $\theta_q = iq^{-1/2}(E-Fk)$ and we put
$${}_\mathbb{F}\mathcal{O}_q(K\backslash U) = \{ y \in {}_\mathbb{F}\mathcal{O}_q(U) : y \triangleleft \theta_q = 0 \}.$$
where the symbol $\triangleleft$ refers to the right action of ${}_\mathbb{F}U_{q}(\mathfrak{u})$ on ${}_\mathbb{F}\mathcal{O}_q(U)$ defined by $y \triangleleft x  =\langle x, y_{(1)} \rangle y_{(2)}$. Then ${}_\mathbb{F}U_{q}(\mathfrak{k})$ and ${}_\mathbb{F}\mathcal{O}_q(K\backslash U)$ are right coideal $(\mathbb{F},\ast)$-subalgebras of ${}_\mathbb{F}U_{q}(\mathfrak{u})^\mathrm{cop}$ and ${}_\mathbb{F}\mathcal{O}_q(U)$ respectively.

\begin{definition}
    The quantized enveloping $\ast$-algebra of $\mathfrak{g}$, denoted by ${}_\mathbb{F}U_{q}(\mathfrak{g})$, is the subalgebra of ${}_\mathbb{F}U_{q}(\mathfrak{g}_\mathbb{C})$ generated by ${}_\mathbb{F}U_{q}(\mathfrak{k})$ and ${}_\mathbb{F}\mathcal{O}_q(K \backslash U)$.
\end{definition}

It is clear that ${}_\mathbb{F}U_{q}(\mathfrak{g})$ is a right coideal $(\mathbb{F},\ast)$-subalgebra of ${}_\mathbb{F}U_{q}(\mathfrak{g}_\mathbb{C})$. The latter is a quantized version of the enveloping $\ast$-algebra of $\mathfrak{g}_\mathbb{C}$ (seen as a real Lie algebra). The subalgebra ${}_\mathbb{F}U_{q}(\mathfrak{g})$ is somehow the good analogue of $U(\mathfrak{g}) \subset U(\mathfrak{g}_\mathbb{C})$ within it. This is explained with more details in \citelist{\cite{DCquantisation}*{below Definition 6.1} \cite{YGE}*{Section 2}}.

One can describe ${}_\mathbb{F}U_{q}(\mathfrak{g})$ in terms of generators and relations. First, one can show that ${}_\mathbb{F}\mathcal{O}_q(K\backslash U)$ is the universal unital $(\mathbb{F},\ast)$-algebra with generators $X_q$ and $Z_q$ subject to the following relations:
\begin{equation}\label{XZ}\begin{gathered}
    Z_q^* = Z_q, \qquad X_qZ_q = q^2Z_qX_q, \\ X_qX_q^*+q^2Z_q^2=1, \qquad X_q^*X_q+q^{-2}Z_q^2=1,
\end{gathered}
\end{equation}
Then ${}_\mathbb{F}U_{q}(\mathfrak{g})$ identifies with the universal unital $(\mathbb{F},\ast)$-algebra with generators $\theta_q$, $X_q$, $Z_q$ subject to the above relations, $\theta_q^*=\theta_q$ and
\begin{equation}\label{thetaXZ}
    \begin{gathered}
        qX_q\theta_q - q^{-1}\theta_q X_q = [2]_qZ_q, \qquad Z_q \theta_q - \theta_q Z_q = X_q-X_q^*.
    \end{gathered}
\end{equation}

One can also define analogues of the $(\mathfrak{g},K)$-modules.
\begin{definition}
    Let $V$ be a ${}_\mathbb{F}U_{q}(\mathfrak{g})$-module. We call $V$ a ${}_\mathbb{F}(\mathfrak{g},K)_q$-module if the action of $\theta_q \in {}_\mathbb{F}U_{q}(\mathfrak{k})$ is diagonalizable with $q$-integer eigenvalues.
\end{definition}

It turns out \cite{YGE}*{§5} that the ${}_\mathbb{F}(\mathfrak{g},K)_q$-modules correspond to the non-degenerate modules over a certain approximately unital $(\mathbb{F}, \ast)$-algebra ${}_\mathbb{F}R_q(\mathfrak{g},K)$, which we thus view as an analogue of the Hecke algebra $R(\mathfrak{g},K)$. It is defined as follows. Let ${}_\mathbb{F}R_q(K)$ be a $(\mathbb{F}, \ast)$-algebra having an $\mathbb{F}$-basis $(e_{q,n})_{n \in \mathbb{Z}}$ of mutually orthogonal projections. Note that ${}_\mathbb{F}R_q(K)$ acts on any ${}_\mathbb{F}(\mathfrak{g},K)_q$-module, with $e_{q,n}$ acting as the projection onto the $[n]_q$-eigenspace for each $n \in \mathbb{Z}$. Then, the multiplication of elements of ${}_\mathbb{F}R_q(K)$ and ${}_\mathbb{F}\mathcal{O}_q(K\backslash U)$ on any ${}_\mathbb{F}(\mathfrak{g},K)_q$-module obey certain universal relations. As a space, ${}_\mathbb{F}R_q(\mathfrak{g},K)$ is the tensor product ${}_\mathbb{F}R_q(K)\otimes{}_\mathbb{F}\mathcal{O}_q(K\backslash U)$. Its $(\mathbb{F}, \ast)$-algebra structure is defined according to the previous relations. We refer to \cite{YGE}*{§2} for a more rigorous definition.

There is also an analogue ${}_\mathbb{F}R_q(\mathfrak{p}, M)$ of the Hecke algebra of the pair $(\mathfrak{p},M)$ and this leads to an adaptation of the parabolic induction to our quantized context. More precisely, one can construct an induction functor ${}_\mathbb{F}\mathrm{ind}_q$ which associates a ${}_\mathbb{F}(\mathfrak{g},K)_q$-module to any ${}_\mathbb{F}R_q(\mathfrak{p}, M)$-module, see \cite{YGE}*{Definition 6.3}. It can be shown \cite{YGE}*{Proposition 6.9} that the characters of ${}_\mathbb{F}R_q(\mathfrak{p}, M)$ are naturally parametrized by $\{-1,1\} \times \mathbb{F}^*$. For any $(\varepsilon,\lambda) \in \{-1,1\} \times \mathbb{F}^*$, let us denote by ${}_\mathbb{F}\mathrm{ind}_q(\varepsilon, \lambda)$ the ${}_\mathbb{F}(\mathfrak{g},K)_q$-module induced by the one-dimensional ${}_\mathbb{F}R_q(\mathfrak{p}, M)$-module corresponding to the pair $(\varepsilon,\lambda)$. Concretely, ${}_\mathbb{F}\mathrm{ind}_q(\varepsilon, \lambda)$ can be realized as a $\lambda$-independent $\mathbb{F}$-vector space ${}_\mathbb{F}\mathcal{O}_q^\varepsilon(K)$, having a basis $(\zeta_{q,n})_{n \in \mathbb{Z}^\varepsilon}$, on which ${}_\mathbb{F}U_{q}(\mathfrak{g})$ acts as follows:
\begin{equation}\label{actionT}
    \begin{gathered}
        \theta \zeta_{q,n} = [n]_q\zeta_{q,n}, \qquad T_{q,n}\zeta_{q,n} = (\lambda + \lambda^{-1})\zeta_{q,n},\\
         T_{q,n}^\pm\zeta_{q,n}= (\lambda q^{1\pm n}-\lambda^{-1}q^{-1\mp n})\zeta_{q,n\pm 2}.
    \end{gathered}
\end{equation}
where
\begin{equation}\label{TXZ}
    T_{q,n} = q^{-1}X_q + qX_q^* + (q^n-q^{-n})Z_q, \qquad T_{q,n}^\pm = q^{\pm n}X_q - q^{\mp n}X_q^* \mp [2]_qZ_q.
\end{equation}
These induced modules are generically simple. The exceptions occur when $\lambda$ is of the form $\sigma q^n$ for $\sigma = \pm 1$ and $n \in \mathbb{Z}^{-\varepsilon}$. In the latter case, if $n\geq 0$ then ${}_\mathbb{F}\mathrm{ind}_q(\varepsilon, \sigma q^n)$ admits two simple submodules, namely ${}_\mathbb{F}D_q^\pm(\sigma,n) = \mathrm{span}_\mathbb{F}(\zeta_{q,\pm m} : m-n \in \mathbb{Z}^{\mathrm{odd}}_{>0})$, which are analogous to the discrete series representations. We refer to \cite{YGE}*{§6} for more precisions on this matter.

From now, let us restrict to the case $\mathbb{F} = \mathbb{C}$. The assumptions on $q$ amount to $q >0$ and $q \neq 1$. The explicit mention to the field in the notations will be omitted in this context. We denote the unit circle by $\mathbb{U}$ and we write
$$\mathbb{U}_+ = \{z \in \mathbb{U} : \Im z \geq 0 \}.$$

One can show that $R_q(\mathfrak{g},K)$ admits a universal enveloping C*-algebra $C^*_q(G)$ which we interpret as a $q$-analogue of the maximal group C*-algebra of $G$. Let us call a $(\mathfrak{g},K)_q$-module unitarizable or unitary if the associated $R_q(\mathfrak{g},K)$-module is. There is a one-one correspondence between the spectrum of $C^*_q(G)$ and the set of equivalence classes of unitarizable simple $(\mathfrak{g},K)_q$-modules. This correspondence is defined as follows: if $V$ is a simple unitarizable $(\mathfrak{g},K)_q$-module then it has a unique invariant inner product up to a scalar and $R_q(\mathfrak{g},K)$ acts on it through bounded operators, its completion $\bar V$ is then an irreducible representation of $C^*_q(G)$; conversely, if $\mathcal{H}$ is an irreducible representation of $C^*_q(G)$ then $R_q(\mathfrak{g},K)\mathcal{H}$ is a simple $(\mathfrak{g},K)_q$-module. We refer to \cite{YGE}*{§6} for more precisions.

Moreover, one can construct \cite{DCDTinvariant} a representation $L^2_q(G)$ of $C^*_q(G)$ which is analogous to the regular representation of $G$. We define the $q$-deformed group C*-algebra $C^*_{q,r}(G)$ as the quotient of $C^*_q(G)$ by the kernel of this representation. By restating \cite{DCinduction}*{Theorem 5.1} in light of \cite{YGE}*{§7}, we see obtain the following.

\begin{proposition} \label{Irrepsqt}
    A complete set of representatives of the spectrum of $C^*_{q,r}(G)$ is given by the completions of the following unitarizable $(\mathfrak{g},K)_q$-modules:
\begin{itemize}
    \item $\mathrm{ind}_q(\varepsilon, \lambda)$ for $\varepsilon\in \{-1,1\}$ and $\lambda \in \mathbb{U}_+$ such that $(\varepsilon, \lambda) \neq (-1,\pm 1)$,
    \item $D_q^{+}(\sigma,n)$ and $D_q^{-}(\sigma,n)$ for $(\sigma, n) \in \{-1,1\} \times \mathbb{Z}_{\geq 0}$.
\end{itemize}
\end{proposition}

For any $(n,m) \in \mathbb{Z}^2$, let us write
\begin{equation}\label{Tnm}
    T_{q,n}^{m} =  e_{q,m}T_{q,n}e_{q,n}, \qquad T_{q,n}^{m,\pm} =  e_{q,m}T_{q,n}^\pm e_{q,n},
\end{equation}
understood as elements of $R_q(\mathfrak{g},K)$. Later, we will need the following basic result.

\begin{lemma} \label{generatorsR_q} The $\ast$-algebra $R_q(\mathfrak{g},K)$ is generated by the collection of elements $T_{q,n}^{n}$, $T_{q,n}^{n+2,+}$, $T_{q,n}^{n-2,-}$, as $n$ ranges over $\mathbb{Z}$.
\end{lemma}

\begin{proof}
    First, note that $X_{q,n}^{m} = e_{q,m}X_qe_{q,n}$ and $Z_{q,n}^{m} = e_{q,m}Z_qe_{q,n}$, as $(n,m)$ ranges over $\mathbb{Z}^2$, generate $R_q(\mathfrak{g},K)$ as a $\ast$-algebra. This is a straightforward consequence of the two following observations: the family $(\sum_{|l|\leq m} e_{q,l})_{m\in \mathbb{N}}$ is an approximate unit of $R_q(\mathfrak{g},K)$ and $\{X_q,Z_q\}$ generates $\mathcal{O}_q(K\backslash U)$ as a $\ast$-algebra. Next, by inverting the formulas (\ref{TXZ}), one can express $X_{q,n}^{m}$ and $Z_{q,n}^{m}$ as linear combinations of $T_{q,n}^{m}$ and $T_{q,n}^{m,\pm}$ for every $(n,m) \in \mathbb{Z}^2$, so the latter generate the $\ast$-algebra $R_q(\mathfrak{g},K)$.
    
    To conclude, just note that $T_{q,n}^{m}$, $T_{q,n}^{m,+}$, $T_{q,n}^{m,-}$ vanish if $m\neq n$, $m\neq n+2$, $m\neq n-2$ respectively. Indeed, $R_q(\mathfrak{g},K)$ embeds into $C^*_q(G)$ and every unitarizable simple $(\mathfrak{g},K)_q$-module can be embedded into an induced module $\mathrm{ind}_q(\varepsilon, \lambda)$ for some $(\varepsilon, \lambda)\in \{-1,1\}\times \mathbb{C}$, on which $T_{q,n}^{m}, \,T_{q,n}^{m,\pm}$ vanish under the above conditions on $n$ and $m$ in view of (\ref{actionT}). \end{proof}

\subsection{The quantization of the motion group}

Let $G_0^\star$ be the transformation Lie groupoid (see \cite{Renault}*{§1.2.a}) of the right action of $K$ on $K \backslash U$ and let $b$ be the base point of $K \backslash U$. It is an amenable groupoid and its C*-algebra $C^*(G_0^\star)$ coincides with the crossed product $K \ltimes C(K\backslash U)$. By comparison, the group C*-algebra of the motion group $G_0$ is naturally isomorphic to $K \ltimes C_0((\mathfrak{g}/\mathfrak{k})\,\hat{})$. As we will see, as the Mackey parameter approaches $0$, the family of quantizations of $G$ converges (at the C*-algebraic level) to the deformation to the normal cone \cite{DebordSkandalis} induced by the embedding $K \hookrightarrow G_0^\star$, where $K$ is identified with the subgroup(oid) of elements of $G_0^\star$ having source and range equal to the base point $b$ of $K\backslash U$. The representation theoretic analogy between $G$ and $G_0$ can then be extended to the $q$-deformations of $G$, provided that we replace $G_0$ by $G_0^\star$.

Due to Theorem \ref{CrossedProduct} and its expression as a crossed product C*-algebra, the description of the irreducible representations of $C^*(G_0^\star)$ follows from the determination of the orbits of $K\backslash U$ under the action of $K$.

Let $\mathcal{O}(K\backslash U) \subset C(K\backslash U)$ be the unital $\ast$-algebra of $U$-finite functions on $K \backslash U$. One can check that $\mathcal{O}(K\backslash U)$ is generated as a unital $\ast$-algebra by the functions $X, \,Z : K\backslash U \to \mathbb{C}$ defined by
$$\left.\begin{pmatrix}
    Z & iX^*\\
    -iX & -Z
\end{pmatrix}\right|_{b\cdot u} = u^*\theta u \qquad (u \in U),$$
where $\theta$ is the matrix defined by (\ref{thetadef}).
They satisfy the relations $Z^* = Z$ and $XX^* + Z^2 = 1$. In fact, $\mathcal{O}(K\backslash U)$ is the universal unital commutative $\ast$-algebra generated by $X$ and $Z$ and subject to these relations. The equality $X^*X + Z^2 = 1$ implies that $\mathcal{O}(K\backslash U)$ admits an enveloping C*-algebra. The latter is $C(K\backslash U)$, since both $\ast$-algebras are commutative and share the same characters.

Let us introduce the following change of variables:
$$T = X+X^*, \qquad T^\pm = X - X^*\mp 2Z.$$
One can check that any $k \in K$ acts on these variables as follows
$$k\cdot T= T, \qquad k\cdot T^\pm = \zeta_{\pm 2}(k)T^\pm.$$
The orbits of $K\backslash U$ under the action of $K$ are thus the level sets of $T$. For every $\lambda$ in the closed upper half-circle $\mathbb{U}_+$, let $p^\star_\lambda$ denote the point of $K\backslash U$ defined by
$X(p^\star_\lambda) = \lambda,\, Z(p^\star_\lambda) = 0$. From the above, it follows that each $K$-orbit is generated by $p^\star_\lambda$ for a unique $\lambda \in \mathbb{U}_+$. The stabilizer of $p^\star_\lambda$ is $K$ if $\lambda \in \{-1, 1\}$, otherwise it is $M = \{-1,1\}$. By Theorem \ref{CrossedProduct}, we get the following.
\begin{proposition}\label{spectrumGstar}
    A list of representatives of the spectrum of $C^*(G_0^\star)$ is
\begin{itemize}
    \item $\mathrm{Ind}(p^\star_\lambda, M, \varepsilon)$ for $\lambda \in \mathbb{U}_+\setminus \{-1,1\}$ and $\varepsilon \in \{-1,1\}$,
    \item the one-dimensional representations $\mathrm{Ind}(p^\star_\sigma, K, n)$ for $\sigma \in \{-1,1\}$ and $n \in \mathbb{Z}$.
\end{itemize}
\end{proposition}

For every $(\lambda,\varepsilon) \in \mathbb{U}_+ \times \{-1,1\}$, the underlying Hilbert space of $\mathrm{Ind}(p^\star_\lambda, M, \varepsilon)$ is $L^2(K)^\varepsilon$. When $\lambda \in \{-1,1\}$ and $n \in \mathbb{Z}^\varepsilon$, the one-dimensional subspace $\mathbb{C}\zeta_n$ is an irreducible subrepresentation of $\mathrm{Ind}(p^\star_\lambda, M, \varepsilon)$ identifying with $\mathrm{Ind}(p^\star_\lambda, K, n)$.

As for the motion group, these representations can be characterized by the action on a dense subspace of an analogue of the enveloping $\ast$-algebra, as follows. By differentiating the action of $K$ on $\mathcal{O}(K\backslash U)$, we get an action $\#$ of $U(\mathfrak{k})$ on $\mathcal{O}(K\backslash U)$. The functor associating to any unital $\ast$-algebra $A$ the set of pairs of unital $*$-algebra morphisms $(\iota : \mathcal{O}(K\backslash U) \to A,\, \kappa: U(\mathfrak{k}) \to A)$ such that for all $f \in \mathcal{O}(K\backslash U)$, $\kappa(\theta)\iota(f) -\iota(f)\kappa(\theta) = \iota(\theta \# f)$, is representable. If $\mathfrak{g}_0^\star$ denotes the Lie algebroid of $G_0^\star$, then we define $U(\mathfrak{g}_0^\star)$ as a universal element of this functor. The motivation for this notation is that $U(\mathfrak{g}_0^\star)$ can be viewed as the enveloping $\mathcal{O}(K\backslash U)$-algebra of the space of "regular" sections of $\mathfrak{g}_0^\star$. One can check that the universal $*$-algebra morphisms from $U(\mathfrak{k})$ and $\mathcal{O}(K\backslash U)$ to $U(\mathfrak{g}_0^\star)$ are injective.

For any representation of $C^*(G_0^\star)$ on a Hilbert space $\mathcal{H}$, the space of $K$-finite vectors of $\mathcal{H}$ inherits a structure of $U(\mathfrak{g}_0^\star)$-module by differentiating the action of $K$ and restricting the action of $C(K\backslash U)$ to $\mathcal{O}(K\backslash U)$. For any point $x \in K\backslash U$, any subgroup $\Gamma \subset K_x$ and any irreducible unitary representation $\pi$ of $\Gamma$, let us denote by $\mathrm{ind}(x,\Gamma,\pi)$ the unitary $U(\mathfrak{g}_0^\star)$-module consisting of the space of $K$-finite vectors of $\mathrm{Ind}(x, \Gamma, \pi)$. If $\varepsilon \in\{-1, 1\}$ and $\lambda \in \mathbb{U}_+$, the underlying space of $\mathrm{ind}(\tilde p_\lambda,M,\varepsilon)$ is $\mathcal{O}^\varepsilon(K)$ and the action of $U(\mathfrak{g}_0^\star)$ on it is characterized by
\begin{gather}\label{actionLT}\theta \zeta_n = n \zeta_n, \qquad T\zeta_n = (\lambda + \lambda^{-1})\zeta_n, \qquad T^\pm \zeta_n = (\lambda-\lambda^{-1})\zeta_{n\pm 2}.\end{gather}

\begin{convention} \label{normalization} In the following, we identify $\mathfrak{g}/\mathfrak{k} \otimes_\mathbb{R} \mathbb{C}$ with the complexified cotangent space $T_b^*(K\backslash U)\otimes_\mathbb{R} \mathbb{C}$ of $K\backslash U$ at $b$ by identifying $\delta X$, $\delta Z$ and $\delta T^\pm$ with the differentials of $X$, $Z$ and $T^\pm$ at $b$. In this picture, $\mathfrak{g}/\mathfrak{k}$ identifies with $i T_b^*(K\backslash U)$. This way, the Pontryagin dual $(\mathfrak{g}/\mathfrak{k})\,\hat{}$ of $\mathfrak{g}/\mathfrak{k}$ identifies with the tangent space $T_b(K\backslash U)$ as follows: to each character $\chi$ of $\mathfrak{g}/\mathfrak{k}$ corresponds the unique vector $v \in T_b(K\backslash U)$ such that $\chi(i\xi) = e^{i\langle \xi,v \rangle}$ for all $\xi \in T_b^*(K\backslash U) = -i\mathfrak{g}/\mathfrak{k}$.
\end{convention}

\section{The continuous field of C*-algebras}

The aim of this third section is to define the continuous field of C*-algebras evoked in the introduction. In order to have a better insight into this construction, we first define an algebraic field for (the analogues of) the universal enveloping $\ast$-algebras. More precisely, we construct an involutive algebra ${}_\mathbb{A}U(\mathfrak{g})$ defined over the ring $\mathbb{A}$ of analytic functions on $\mathbb{R}_+^*\times \mathbb{R}$ and we show that it specializes at each $(q,t) \in \mathbb{R}_+^*\times \mathbb{R}$ according to the following table.

\begin{center}
\begin{tabular}{ |c|c|c| } 
 \cline{2-3}
\multicolumn{1}{c|}{} & $q\neq 1$ & $q = 1$ \\
 \hline
 $t \neq 0$ & $U_{q^t}(\mathfrak{g})$ & $U(\mathfrak{g})$ \\
 \hline
 $t = 0$ & $U(\mathfrak{g}_0^\star)$ & $U(\mathfrak{g}_0)$\\
 \hline
\end{tabular}
\end{center}
Although we do not prove it explicitly, the continuous field of C*-algebras is then constructed so that the elements of ${}_\mathbb{A}U(\mathfrak{g})$ give continuous vector fields of unbounded multipliers.

\subsection{The algebraic quantum Mackey field} In this subsection, we define the algebra ${}_\mathbb{A}U(\mathfrak{g})$ and show how some of its modules specialize at points $(q,t)$ to the various irreducible representations introduced in Section 2. We emphasize that this part is mainly motivational and may thus be skipped on a first reading.

\begin{vocabulary}
    Given a ring morphism $\varphi : A_1 \to A_2$ between commutative rings and $M_j$ an $A_j$-module for each $j \in \{1,2\}$, an additive map $f : M_1 \to M_2$ is said to be $\varphi$-linear if $f(am) = \varphi(a)f(m)$ for every $(a,m) \in A \times M$.

Let $A$ be a commutative $\ast$-algebra, let $\varphi : A \to \mathbb{C}$ be a $\ast$-algebra morphism and let $B$ be a unital $(A,\ast)$-algebra, that is, a unital $\ast$-algebra equipped with a unital morphism of $\ast$-algebras $\iota :A \to Z(B)$, where $Z(B)$ denotes the center of $B$. By a specialization of $B$ at $\varphi$, we mean a universal element of the functor associating to any unital $\ast$-algebra $C$ the set of $\varphi$-linear $\ast$-algebra morphisms $B \to C$. Let $\varphi_B : B \to B_\varphi$ be such a specialization and $M$ be any $B$-module. The functor associating to any $B_\varphi$-module $N$ the set of $\varphi_B$-linear maps $M \to N$ is representable. A universal element of that functor is said to be a specialization of $M$ at $\varphi$.
\end{vocabulary}

Let $\mathbb{A}$ be the $\ast$-algebra of real analytic complex-valued functions on $\mathbb{R}_+^*\times \mathbb{R}$. Its fraction field, denoted by $\mathbb{F}$, is an involutive extension of $\mathbb{C}$. We denote by\footnote{We use the archaic greek letter \textit{koppa} $\cpa$ to denote the coordinate function of the parameter $q$.} $(\cpa,\tau)$ the canonical coordinates of $\mathbb{R}_+^*\times \mathbb{R}$. Note that the element $\cpa^\tau \in \mathbb{F}^\times$ satisfies the conditions needed to define a deformed universal enveloping algebra ${}_\mathbb{F}U_{\cpa^\tau}(\mathfrak{g})$, see Section 2.3, and we recall that we have generators $X_{\cpa^\tau}$, $Z_{\cpa^\tau}$, $\theta_{\cpa^\tau}$ as in (\ref{XZ}), (\ref{thetaXZ}).

\begin{definition}
We define ${}_\mathbb{A}U(\mathfrak{g})$ as the $(\mathbb{F},\ast)$-subalgebra of ${}_\mathbb{F}U_{\cpa^\tau}(\mathfrak{g})$ generated by the elements
$$\theta_\mathbb{A} = \theta_{\cpa^\tau}, \qquad x_\mathbb{A} = \eta^{-1}(X_{\cpa^\tau}-1), \qquad z_\mathbb{A} = \eta^{-1} Z_{\cpa^\tau},$$
where $\eta$ is defined as the element of $\mathbb{A}$ such that $\cpa^\tau - \cpa^{-\tau} = \tau \eta$.
\end{definition}

By applying the exact same techniques as in \cite{YGE}*{§4}, we get the following result.
\begin{proposition}
    ${}_\mathbb{A}U(\mathfrak{g})$ is a universal unital $(\mathbb{F},\ast)$-algebra with generators $\theta_\mathbb{A}$, $x_\mathbb{A}$, $z_\mathbb{A}$ and relations
    \begin{equation}
        \begin{gathered}
            \cpa^\tau x_\mathbb{A}\theta _\mathbb{A} - \cpa^{-\tau} \theta_\mathbb{A} x_\mathbb{A} = [2]_{\cpa^\tau} z_\mathbb{A} - \tau \theta_\mathbb{A}, \qquad
            z_\mathbb{A} \theta_\mathbb{A} - \theta_\mathbb{A} z_\mathbb{A} = x_\mathbb{A} - x_\mathbb{A}^*,\\
            \cpa^\tau z_\mathbb{A} x_\mathbb{A} - \cpa^{-\tau}x_\mathbb{A}z_\mathbb{A} = -\tau z_\mathbb{A},\qquad
            x_\mathbb{A} x_\mathbb{A}^* + \cpa^{2\tau}z_\mathbb{A}^2 = x_\mathbb{A}^*x_\mathbb{A} + \cpa^{-2\tau}z_\mathbb{A}^2,\\
            x_\mathbb{A} +x_\mathbb{A}^*+\eta(x_\mathbb{A} x_\mathbb{A}^* + \cpa^{2\tau}z_\mathbb{A}^2) = 0.
        \end{gathered}
    \end{equation}
\end{proposition}

The characters of the $\ast$-algebra $\mathbb{A}$ are the evaluations at points $(q,t) \in \mathbb{R}^*_+\times \mathbb{R}$. The following proposition describes the corresponding specializations of ${}_\mathbb{A}U(\mathfrak{g})$.

\begin{proposition}
    Let $(q,t) \in \mathbb{R}_+^*\times \mathbb{R}$ and let $h = 2(\log q)$. The specialization of ${}_\mathbb{A}U(\mathfrak{g})$ at $(q,t)$, which we denote by $\mathrm{ev}_{q,t} : {}_\mathbb{A}U(\mathfrak{g}) \to U_{q,t}(\mathfrak{g})$, is given by the following table.
    \begin{center}
        \begin{tabular}{ |c|c|c|c|c| } 
         \hline
        Value of $(q,t)$ & $U_{q,t}(\mathfrak{g})$ & $\mathrm{ev}_{q,t}(x_\mathbb{A})$ & $\mathrm{ev}_{q,t}(z_\mathbb{A})$ & $\mathrm{ev}_{q,t}(\theta_\mathbb{A})$ \\
         \hline\hline
         $q\neq 1, \,t \neq 0$ & $U_{q^t}(\mathfrak{g})$ & $\eta(q,t)^{-1}(X_{q^t}-1)$ & $\eta(q,t)^{-1}Z_{q^t}$ & $\theta_{q^t}$ \\
         \hline\hline
         $q\neq 1, \,t = 0$ & $U(\mathfrak{g}_0^\star)$ & $h^{-1}(X-1)$ & $h^{-1}Z$ & $\theta$ \\
         \hline\hline
         $q=1, \,t \neq 0$ & $U(\mathfrak{g})$ & $tH/2$ & $itE$ & $\theta$ \\
         \hline\hline
         $q=1, \,t = 0$ & $U(\mathfrak{g}_0)$ & $\delta X$ & $\delta Z$ & $\theta$ \\
         \hline
        \end{tabular}
    \end{center}
\end{proposition}
\vspace{0.12cm}
This follows from the characterization of ${}_\mathbb{A}U(\mathfrak{g})$ in the last proposition and the description of the algebras $U_q(\mathfrak{g})$, $U(\mathfrak{g}_0^\star)$, $U(\mathfrak{g})$, $U(\mathfrak{g}_0)$ in terms of generators and relations given in the previous section.

One can also describe explicitly the specializations of the analogues of the parabolically induced modules from Section 2.3. Let $\varepsilon \in \{-1,1\}$ and $\lambda$ be any invertible element of $\mathbb{A}$ such that $\lambda = 1$ wherever $\cpa = 1$. As explained in the last section, the underlying space of the ${}_\mathbb{F}(\mathfrak{g},K)_{\cpa^\tau}$-module ${}_\mathbb{F}\mathrm{ind}_{\cpa^\tau}(\varepsilon, \lambda)$ is an $\mathbb{F}$-vector space $_{\mathbb{F}}\mathcal{O}_{\cpa^\tau}^\varepsilon (K)$, which is independent of $\lambda$ and has a basis $(\zeta_{\cpa^\tau\!,n})_{n \in \mathbb{Z}^\varepsilon}$. Let us denote by ${}_\mathbb{A}\mathcal{O}^{\varepsilon}(K)$ the $\mathbb{A}$-submodule of $_{\mathbb{F}}\mathcal{O}_{\cpa^\tau}^\varepsilon (K)$ spanned by the elements of this basis.

\begin{lemma}
    The action of ${}_\mathbb{A}U(\mathfrak{g})$ on ${}_\mathbb{F}\mathrm{ind}_{\cpa^\tau}(\varepsilon, \lambda)$ stabilizes ${}_\mathbb{A}\mathcal{O}^{\varepsilon}(K)$.
\end{lemma}

\begin{proof}
    For every $n \in \mathbb{Z}$, let us write
    \begin{gather}\label{actionsA} \begin{gathered} s_{\mathbb{A},n} = \cpa^{-\tau}x_\mathbb{A} + \cpa^\tau x_\mathbb{A}^* + (\cpa^{n\tau} - \cpa^{-n\tau})z_\mathbb{A} = \eta^{-1}(T_{\cpa^\tau\!,n} - [2]_{\cpa^\tau}),\\
    s_{\mathbb{A},n}^\pm = \cpa^{\pm n \tau}x_\mathbb{A} - \cpa^{\mp n \tau}x_\mathbb{A}^* \mp[2]_{\cpa^\tau}z_\mathbb{A} \pm \tau [n]_{\cpa^\tau} = \eta^{-1}T_{\cpa^\tau\!,n}^\pm.
    \end{gathered}\end{gather}
    Let us denote by ${}_\mathbb{A}V$ the $\mathbb{A}$-submodule of ${}_\mathbb{A}U(\mathfrak{g})$ spanned by $1, x_\mathbb{A}, x^*_\mathbb{A}, z_\mathbb{A}$. One can check that for each $n \in \mathbb{Z}$, we have ${}_\mathbb{A}V = \mathrm{span}_\mathbb{A}(1,s_{\mathbb{A},n}, s_{\mathbb{A},n}^+, s_{\mathbb{A},n}^-)$. Now, using the formulas (\ref{actionT}), we get
    \begin{equation}\label{s+-}\begin{gathered}
        s_{\mathbb{A},n} \zeta_{\cpa^\tau\!,n} = \eta^{-1}\left(\lambda + \lambda^{-1}-[2]_{\cpa^\tau}\right)\zeta_{\cpa^\tau\!,n}, \\ s_{\mathbb{A},n}^\pm \zeta_{\cpa^\tau\!,n} = \eta^{-1}\left(\lambda \cpa^{\tau(1\pm n)} - \lambda^{-1} \cpa^{-\tau(1\pm n)}\right)\zeta_{\cpa^\tau\!,n \pm 2}.
    \end{gathered}
    \end{equation}
    From the very expression of $\eta$, it follows that $\eta/(\cpa -1) \in \mathbb{A}^\times$. Hence, for any $f \in \mathbb{A}$, if $f = 1$ wherever $\cpa = 1$, then $\eta^{-1}f \in \mathbb{A}$. Since $\lambda = 1$ on $\cpa^{-1}(1)$, the above expressions (\ref{s+-}) show that
    $${}_\mathbb{A} V \zeta_{\cpa^\tau\!,n} \subset \mathrm{span}_\mathbb{A}\!\left(\zeta_{\cpa^\tau\!,m} : m \in \{n-2, n, n+2\}\right).$$
    Since ${}_\mathbb{A}U(\mathfrak{g})$ is generated as an $\mathbb{A}$-algebra by the elements of ${}_\mathbb{A} V$ and by $\theta_\mathbb{A}$, which clearly preserves ${}_\mathbb{A}\mathcal{O}^{\varepsilon}(K)$, this concludes the proof.
\end{proof}

Let us denote by ${}_\mathbb{A}\mathrm{ind}(\varepsilon, \lambda)$ the ${}_\mathbb{A}U(\mathfrak{g})$-module ${}_\mathbb{A}\mathcal{O}^{\varepsilon}(K) \subset {}_\mathbb{F}\mathrm{ind}_{\cpa^\tau}(\varepsilon, \lambda)$. The next proposition summarizes how this module specializes at each point of $\mathbb{R}_+^*\times \mathbb{R}$.

\begin{notation}
    We write $\theta_1 = \theta$, $\mathcal{O}_1^\varepsilon(K) = \mathcal{O}^\varepsilon(K)$ and $\zeta_{1,n} = \zeta_n$ for every $n \in \mathbb{Z}$.
\end{notation}

\begin{proposition}
    We fix $(q,t) \in \mathbb{R}_+^*\times \mathbb{R}$. Let $\mathrm{ev}_{q,t}^\varepsilon : {}_\mathbb{A}\mathcal{O}^{\varepsilon}(K) \to \mathcal{O}_{q^t}^\varepsilon(K)$ be the unique map which is linear with respect to the evaluation of $\mathbb{A}$ at $(q,t)$ and such that $\mathrm{ev}_{q,t}^\varepsilon(\zeta_{\cpa^\tau\!,n}) = \zeta_{q^t\!,n}$ for all $n \in \mathbb{Z}^\varepsilon$. Let $\mathrm{ind}_{q,t}(\varepsilon,\lambda)$ be the unique $U_{q,t}(\mathfrak{g})$-module whose underlying space is $\mathcal{O}_{q^t}^\varepsilon(K)$ and such that $\mathrm{ev}_{q,t}^\varepsilon : {}_\mathbb{A}\mathrm{ind}(\varepsilon,\lambda) \to \mathrm{ind}_{q,t}(\varepsilon,\lambda)$ is a specialization of ${}_\mathbb{A}\mathrm{ind}(\varepsilon,\lambda)$ at $\mathrm{ev}_{q,t}$. The nature of $\mathrm{ind}_{q,t}(\varepsilon,\lambda)$ is given by the following table.
        \begin{center}
        \begin{tabular}{ |c|c|c| } 
         \cline{2-3}
        \multicolumn{1}{c|}{} & $q\neq 1$ & $q = 1$ \\
         \hline
         $t \neq 0$ & $\mathrm{ind}_{q^t}(\varepsilon, \lambda(q,t))$ & $\mathrm{ind}(\varepsilon, t^{-1} \partial_\cpa\lambda(1,t))$ \\
         \hline
         $t = 0$ & $\mathrm{ind}(p^\star_{\lambda(q,0)}, M, \varepsilon)$ & $\mathrm{ind}(p_{\partial_\cpa\lambda(1,0)}, M, \varepsilon)$\\
         \hline
        \end{tabular}
        \end{center}
\end{proposition}

\begin{proof} As an $\mathbb{A}$-algebra, ${}_\mathbb{A}U(\mathfrak{g})$ is generated by $\theta_\mathbb{A}, 1, x_\mathbb{A}, x_\mathbb{A}^*, z_\mathbb{A}$. For every $n \in \mathbb{Z}$, the latter can be expressed as $\mathbb{A}$-linear combinations of $\theta_\mathbb{A},1,s_{\mathbb{A},n}, s_{\mathbb{A},n}^+, s_{\mathbb{A},n}^-$, defined by (\ref{actionsA}). Thus, to prove that $\mathrm{ind}_{q,t}(\varepsilon,\lambda)$ coincides with the $U_{q,t}(\mathfrak{g})$-module of the above table, it is enough to check that the images of $\theta_\mathbb{A},1,s_{\mathbb{A},n}, s_{\mathbb{A},n}^+, s_{\mathbb{A},n}^-$ through $\mathrm{ev}_{q,t}$ act in the same way on every $\zeta_{q^t\!,n}$ for both modules. In order to check that, we will need the following table, where $h = \eta(q,0) = 2 \log q$.
\begin{center}
        \begin{tabular}{ |c|c|c|c|c| } 
         \hline
        Value of $(q,t)$ & $\mathrm{ev}_{q,t}(s_{\mathbb{A},n})$ & $\mathrm{ev}_{q,t}(s^\pm_{\mathbb{A},n})$ \\
         \hline\hline
         $q\neq 1, \,t \neq 0$ & $\eta(q,t)^{-1}(T_{q^t} - [2]_{q^t})$ & $\eta(q,t)^{-1}T_{q^t}^\pm$ \\
         \hline\hline
         $q\neq 1, \,t = 0$ & $h^{-1}(T-2)$ & $h^{-1}T^\pm$ \\
         \hline\hline
         $q=1, \,t \neq 0$ & $0$ & $t(H\mp i(E+F) + n-\theta)$ \\
         \hline\hline
         $q=1, \,t = 0$ & $0$ & $\delta T^\pm$ \\
         \hline
        \end{tabular}
    \end{center}
Moreover, the action of these elements on every $\zeta_{q^t\!,n} \in \mathrm{ind}_{q,t}(\varepsilon,\lambda)$ is obtained by evaluating (\ref{s+-}) at $(q,t)$:
\begin{gather*}
    \mathrm{ev}_{q,t}(s_{\mathbb{A},n})\zeta_{q^t\!,n} = \kappa_{n}(q,t) \zeta_{q^t\!,n},\\
    \mathrm{ev}_{q,t}(s^\pm_{\mathbb{A},n})\zeta_{q^t\!,n} = \kappa_{n}^\pm(q,t) \zeta_{q^t\!,n\pm2},
\end{gather*}
where the value of the constants $\kappa_{n}(q,t)$, $\kappa_{n}^\pm(q,t)$ are given by
\begin{center}
        \begin{tabular}{ |c|c|c|c|c| } 
         \hline
        Value of $(q,t)$ & $\kappa_{n}(q,t)$ & $\kappa_{n}^\pm(q,t)$ \\
         \hline\hline
         $q\neq 1, \,t = 0$ & $h^{-1}(\lambda(q,0) + \lambda(q,0)^{-1}-2)$ & $h^{-1}(\lambda(q,0) - \lambda(q,0)^{-1})$ \\
         \hline\hline
         $q=1, \,t \neq 0$ & $0$ & $\partial_\cpa \lambda(1,t) + t(1\pm n)$ \\
         \hline\hline
         $q=1, \,t = 0$ & $0$ & $\partial_\cpa \lambda(1,0)$ \\
         \hline
        \end{tabular}
    \end{center}
if $q= 1$ or $t= 0$; otherwise they can be red from (\ref{s+-}) directly. To conclude the proof, just compare these tables with the formulas (\ref{actioninducedclassical}), (\ref{actionG0}), (\ref{actionT}), (\ref{actionLT}).
\end{proof}

A similar statement can be made for the analogues of the discrete series. Let us fix $n \in \mathbb{Z}_{>0}$ and write $\varepsilon = (-1)^{n+1}$. One can check that $${}_\mathbb{A}\mathcal{O}^{\varepsilon}(K) \cap {}_\mathbb{F}D_{\cpa^\tau}^+(1,n)\qquad \text{and}\qquad {}_\mathbb{A}\mathcal{O}^{\varepsilon}(K) \cap {}_\mathbb{F}D_{\cpa^\tau}^-(1,n)$$ are stable by the action of ${}_\mathbb{A}U(\mathfrak{g})$. Let us denote by ${}_\mathbb{A}D^+(n)$ and ${}_\mathbb{A}D^-(n)$ the resulting ${}_\mathbb{A}U(\mathfrak{g})$-modules. Next, for every $q\in \mathbb{R}_+^*$, we write
$$D_{q,t}^\pm(n) = \begin{cases}
    D_{q^t}^\pm(1,n) & (q\neq 1)\\
    D^\pm (n) & (q= 1)
\end{cases}$$
when $t\neq 0$ and
$$D_{q,0}^\pm(n) = \begin{cases}
    \bigoplus_{m \in \pm\mathbb{Z}^\varepsilon_{>n}} \mathrm{ind}(p^\star_1,K,n) \subset \mathrm{ind}(p^\star_1, M, \varepsilon) & (q\neq 1)\\
    \bigoplus_{m \in \pm \mathbb{Z}^\varepsilon_{>n}} \mathrm{ind}(p_0,K,n) \subset \mathrm{ind}(p_0, M, \varepsilon) & (q=1)
\end{cases}$$

Applying the previous proposition to the case $\lambda = \cpa^{n\tau}$, we obtain the following.

\begin{proposition}
    For every $(q,t) \in \mathbb{R}_+^* \times \mathbb{R}$, the restrictions of $\mathrm{ev}_{q,t}^\varepsilon$ to the $\mathbb{A}$-span of $\{\zeta_{\cpa^\tau\!,m} : m \in \mathbb{Z}^\varepsilon_{\pm m >n}\}$, viewed as $\mathrm{ev}_{q,t}$-linear maps
    $${}_\mathbb{A}D^\pm(n) \to D_{q,t}^\pm(n),$$
    define specializations of ${}_\mathbb{A}D^\pm(n)$ at $(q,t)$.
\end{proposition}

\begin{remark}
    The other analogues of the discrete series ${}_\mathbb{F}D^\pm_{\cpa^\tau}(-1,n)$ do not fit in the above picture, because they do not admit classical limits as $q \to 1$. We could however develop the same ideas by restricting to analytic functions defined on $(0,1) \times \mathbb{R}$.
\end{remark}

\subsection{Construction of the continuous field} The previous results suggest that the Hilbert completions of the modules $\mathrm{ind}_{q,t}\!\left(\varepsilon, \cpa^\lambda\right)$, for fixed $(\varepsilon, \lambda) \in\{-1,1\}\times i\mathbb{R}$, and $D^\pm_{q,t}(n)$, for a given $n \in \mathbb{Z}_{>0}$, should form continuous families of representations of a certain continuous field of C*-algebras $C^*_r(\mathbf{G})$ over $\mathbb{R}_+^* \times \mathbb{R}$ with fibers
\begin{equation}\label{fibers} C^*_r(\mathbf{G})_{q,t} = \begin{cases}
    C^*_{q^t,r}(G) &\text{($q\neq 1$, $t\neq 0$)}\\
    C^*(G_0^\star) & \text{($q\neq 1$, $t= 0$)}\\
    C^*_r(G) & \text{($q =1$, $t \neq 0$)}\\
    C^*(G_0) & \text{($q =1$, $t =0$)}.
\end{cases}\end{equation}
The aim of the present subsection is to construct this continuous field.

Let us outline our strategy. First, we define a certain locally compact Hausdorff space $S$ and equip it with an open continuous map $\mathcal{L} : S \to \mathbb{R}_+^*\times \mathbb{R}$, called the location map. Roughly, the fiber $S_{q,t}$ of $\mathcal{L}$ over a base point $(q,t)$ parametrizes the (analogues) of the principal and discrete series associated to the deformation of $G$ with parameter $(q,t)$. Then, we define a locally constant continuous field of Hilbert spaces $\mathcal{H}$ over $S$ and a constraint $\mathscr{C}$ on it. The Fourier transform will give rise to an embedding of C*-algebras $C^*_r(\mathbf{G})_{q,t} \hookrightarrow \mathfrak{K}_\mathscr{C}(\mathcal{H})_{|S_{q,t}}$ for each $(q,t)\in \mathbb{R}_+^*\times \mathbb{R}$. Finally, the continuous field of reduced C*-algebras $C^*_r(\mathbf{G})$ is defined so that these embeddings identify continuous vector fields of the latter with continuous vector fields of $\mathcal{L}_\ast\mathfrak{K}_\mathscr{C}(\mathcal{H})$ (the pushforward of $\mathfrak{K}_\mathscr{C}(\mathcal{H})$ by $\mathcal{L}$, see Proposition \ref{pushforward}).

Let us begin with the definitions of $S$ and $\mathcal{L}$. We define $S$ as the topological disjoint union
$$S = \left(\,\bigcup_{\varepsilon \in\{-1,1\}}\!\!\!\mathrm{Pri}_{\varepsilon}\right)\cup\left(\bigcup_{\substack{\sigma \in \{-1,1\}\\n\in \mathbb{Z}_{\geq 1}}}\!\!\!\mathrm{Dis}_{\sigma, n}^+ \cup \mathrm{Dis}_{\sigma,n}^-\right),$$
where
\begin{itemize}
    \item for every $n\in \mathbb{Z}$, the spaces $\mathrm{Dis}_{\sigma,n}^\pm$ are copies of $\mathbb{R}_+^* \times \mathbb{R}$ if $\sigma = 1$, copies of $(\mathbb{R}_+^*\setminus\{1\})\times \mathbb{R}$ if $\sigma = -1$,
    \item for each $\varepsilon \in \{-1,1\}$, the space $\mathrm{Pri_{\varepsilon}}$ is a copy of the deformation to the normal cone relative to the inclusion $\{1\}\times \mathbb{R} \times \mathbb{U}_+ \subset \mathbb{R}_+^* \times \mathbb{R} \times \mathbb{U}_+$, that is, the     set $(\mathbb{R}_+^*\setminus\{1\})\times \mathbb{R} \times \mathbb{U}_+ \bigsqcup\, \{1\} \times \mathbb{R} \times i\mathbb{R}_+$ equipped with the unique topology satisfying:
        \begin{enumerate}[label = (\arabic*)]
            \item the subspace $(\mathbb{R}_+^*\setminus\{1\})\times \mathbb{R} \times \mathbb{U}_+$ is open and inherits the usual topology,
            \item the map $$(q,t,\lambda) \mapsto \begin{cases}
        (q,t, q^\lambda) & \text{if $q\neq 1$}\\ (1,t,\lambda) &\text{if $q= 1$}
    \end{cases}$$ is a homeomorphism of $\{(q,t,\lambda) \in \mathbb{R}_+^*\times \mathbb{R} \times i\mathbb{R}_+ : |\lambda \log q|<\pi\}$ onto the subspace $(\mathbb{R}_+^*\setminus\{1\})\times \mathbb{R} \times (\mathbb{U}_+ \setminus\{-1\}) \bigsqcup\, \{1\} \times \mathbb{R} \times i\mathbb{R}_+$.
        \end{enumerate}
\end{itemize}
Each of these spaces is naturally equipped with an open continuous map with values in $\mathbb{R}_+^*\times \mathbb{R}$, namely the natural inclusion for $\mathrm{Dis}_{\sigma, n}^\pm$ and the projection onto the first two coordinates for $\mathrm{Pri_{\varepsilon}}$. Assembling these maps yields a continuous open surjection $\mathcal{L} : S \to \mathbb{R}_+^*\times \mathbb{R}$. If $s \in S$, we call $\mathcal{L}(s)$ the location of $s$. For every $(q,t) \in \mathbb{R}_+^*\times \mathbb{R}$, we denote by $S_{q,t}$ the preimage of $(q,t)$ by $\mathcal{L}$. More generally, if $D$ is a subspace of $\mathbb{R}_+^*\times \mathbb{R}$, we write $S_D = \mathcal{L}^{-1}(D)$. Moreover, by a slight abuse of notation, we will consider $\cpa$ and $\tau$ as coordinates on $S$ through $\mathcal{L}$.

\begin{definition}
    An element of $S$ is said to be a continuous parameter if it belongs to $\cup_{\varepsilon \in \{-1,1\}}\mathrm{Pri}_\varepsilon$, otherwise we call it a discrete parameter. If $s \in \mathrm{Dis}^\pm_{\sigma,n}$ for some $(\sigma, n) \in \{-1,1\}\times \mathbb{Z}_{>0}$, then $(n,\pm)$ is called the order of $s$.
\end{definition}

We define a functions $\Lambda$ and $e$ on $S$ as follows.
\begin{gather*}\Lambda(s) = \begin{cases}
    \lambda & \text{($s = (q,t,\lambda) \in \mathrm{Pri}_{\varepsilon},\, q \neq 1$)}\\
    \lambda/t &\text{($s = (1,t,\lambda)\in \mathrm{Pri}_{\varepsilon}, \,t\neq 0$)}\\
    \lambda/t &\text{($s = (1,0,\lambda)\in \mathrm{Pri}_{\varepsilon})$}\\
    \sigma q^{nt} & \text{($s = (q,t) \in \mathrm{Dis}_{\sigma,n}^\pm, \, q\neq 1$)}\\
    n & \text{($s = (1,t) \in \mathrm{Dis}_{1,n}^\pm, \,t\neq 0$})\\
    0 & \text{($s = (1,0) \in \mathrm{Dis}_{1,n}^\pm$})
\end{cases},\qquad
e(s) = \begin{cases}
    \varepsilon & \text{($s \in \mathrm{Pri}_\varepsilon$)}\\
    (-1)^{n+1} & \text{($s \in \mathrm{Dis}_{\sigma,n}^\pm$).}
\end{cases}\end{gather*}
As explained before, $S$ can be viewed as a parameter space for irreducible representations. The function $\Lambda$ then plays the role of the infinitesimal character. For a given $s \in S$, we call $e(s)$ the parity of $s$ and we say that $s$ is even (respectively odd) if $e(s) = 1$ (respectively $e(s) = -1$). Moreover, for any $s \in S$, we write
$$\mathbb{Z}(s) = \begin{cases}
    \mathbb{Z}^\varepsilon &\text{($s \in \mathrm{Pri}_\varepsilon$)}\\
    \pm \mathbb{Z}^{e(s)}_{>n} &\text{($s \in \mathrm{Dis}_{\sigma,n}^\pm$)},
\end{cases}$$
which we somehow interpret as the set of $K$-types of $s$.

Now let us define the continuous field $\mathcal{H}$. For each $q \in \mathbb{R}_+^*$ and $\varepsilon \in \{-1,1\}$, let us denote by $L^2_q(K)^\varepsilon$ the completion of $\mathcal{O}_q^\varepsilon(K)$ with respect to the unique inner product $(\cdot|\cdot)_{0,q}$ such that
$$(\zeta_{q,n} | \zeta_{q,m})_{0,q} = \frac{2\delta_{n,m}}{q^n+q^{-n}} \qquad (n,m \in \mathbb{Z}^\varepsilon).$$
For any $(q,n) \in \mathbb{R}_+^* \times\mathbb{Z}_{\geq 0}$, if $\varepsilon = (-1)^{n+1}$, then let us denote by $\mathcal{D}_{q,n,\pm}$ the Hilbert completions of the subspaces
$\mathrm{span}(\zeta_{q,m} : m \in \pm \mathbb{Z}_{>n}^{\varepsilon})$
of $\mathcal{O}_q^\varepsilon(K)$, equipped with the inner product $(\cdot |\cdot)_{n,q}$ characterized by
\begin{gather*}
(\zeta_{q,m}| \zeta_{q,m'})_{n,q} = \frac{2\delta_{m,m'}}{q^{m}+q^{-m}}\prod_{\substack{l = 1 \\l \text{ odd}}}^{|m|-n} \frac{[l-1]_q}{[l-1+2n]_q} \quad (m \in \mathbb{Z}_{>0}^{\mathrm{odd}}).\end{gather*}
We define $\mathcal{H}$ as the unique continuous field of Hilbert spaces over $S$ satisfying:
\begin{itemize}
    \item $\mathcal{H}_s = L^2_{\cpa^\tau(s)}(K)^{e(s)}$ for every continuous parameter $s$,
    \item $\mathcal{H}_s = \mathcal{D}_{\cpa^\tau(s),n,\pm}$ for every discrete parameter $s$ of order $(n,\pm)$,
    \item for every $n \in \mathbb{Z}$, the vector field $s \in S\mapsto \mathbf{1}_{\mathbb{Z}(s)}(n) \zeta_{\cpa^\tau(s),n}$ is continuous.
\end{itemize}
Since it is trivial on each connected component, $\mathcal{H}$ is locally trivial. In order to simplify the notations, for every $s \in S$, we denote by $(\cdot |\cdot)_s$ and $\Vert \cdot\Vert_s$ the inner product and norm on $\mathcal{H}_s$.

We define the constraint $\mathscr{C}$ on $\mathcal{H}$ as follows. The decomposition of the fibers of $\mathcal{H}$ with respect to $\mathscr{C}$ is given by:
\begin{itemize}
    \item $\mathcal{H}_s = \mathcal{D}_{\cpa^\tau(s),0,+} \bigoplus \mathcal{D}_{\cpa^\tau(s),0,-}$ if $s$ is an odd continuous parameter such that $\tau(s) \neq 0$ and $\Lambda(s) \in \mathbb{R}$,
    \item $\mathcal{H}_s = \bigoplus_{n \in \mathbb{Z}(s)} \mathbb{C}\zeta_{n}$ if $\tau(s) = 0$ and $\Lambda(s) \in \mathbb{R}$,
    \item for all other $s$, the decomposition is trivial.\end{itemize}

The embeddings $C^*_r(\mathbf{G})_{q,t} \hookrightarrow \mathfrak{K}_\mathscr{C}(\mathcal{H})_{|S_{q,t}}$ will be isomorphisms for $t\neq 0$, but not at $t=0$, due to the degeneracy of the principal series as $t\to 0$ (see the previous subsection and the remarks after Proposition \ref{spectrumGstar}). In order to prove that the continuous field $C^*_r(\mathbf{G})$ is well-defined near the edge $t=0$, we will need to characterize the image of these embeddings. For that purpose, we introduce some extra material. For every $(\varepsilon,n) \in \{-1, 1\}\times\mathbb{Z}^{-\varepsilon}_{> 0}$ and $m\in \pm\mathbb{Z}_{>n}^\varepsilon$, let $J_{n,\pm}$ denote the surjective linear operator $L^2(K)^\varepsilon \to \mathcal{D}_{1,n,\pm}$ defined by
    $$J_{n,\pm}(\zeta_m) = \mathbf{1}_{\mathbb{Z}^\varepsilon_{>n}}(\pm m) \zeta_m \qquad (m \in \mathbb{Z}^\varepsilon).$$
For any $(q,n) \in \mathbb{R}_+^*\times\mathbb{Z}_{>0}$, for every continuous parameter $s \in S_{q,0}$ and every discrete parameter $s' \in S_{q,0}$ of order $(n,\pm)$ such that $\Lambda(s) = \Lambda(s')$ and $e(s) = e(s')$, we write $J(s',s)$ for the operator $J_{n,\pm} : \mathcal{H}_s \to \mathcal{H}_{s'}$. If $V$ is any locally closed subspace of $\mathbb{R}^*_+\times\mathbb{R}$, a vector field $\sigma$ of $\mathfrak{K}_\mathscr{C}(\mathcal{H})_{|S_V}$ is said to be $J$-equivariant if $J(s',s) \sigma_s = \sigma_{s'}J(s',s)$ for every pair $(s,s') \in S_V\times_VS_V$ as above. We denote by $\mathfrak{K}_\mathscr{C}(\mathcal{H})_{|S_V, J}$ the C*-algebra of continuous $J$-equivariant vector fields of $\mathfrak{K}_\mathscr{C}(\mathcal{H})_{|S_V}$ vanishing at infinity. Assuming that $\tau(V) = \{0\}$, denoting by $S_{V}^c$ the set of continuous parameters of $S_V$, one can check that the morphism $\mathfrak{K}_\mathscr{C}(\mathcal{H})_{|S_{V}} \to \mathfrak{K}_\mathscr{C}(\mathcal{H})_{|S_{V}^c}$ of restriction to $S_{V}^c$ induces a C*-algebra isomorphism
\begin{equation}\label{Jequ}\mathfrak{K}_\mathscr{C}(\mathcal{H})_{|S_{V}, J} \longrightarrow \mathfrak{K}_\mathscr{C}(\mathcal{H})_{|S_{V}^c}.\end{equation}
Moreover, we have the following.

\begin{lemma}\label{subfieldJ}
    There exists a unique continuous subfield of C*-algebras of $\mathcal{L}_*\mathfrak{K}_\mathscr{C}(\mathcal{H})$ whose fiber at each $(q,t) \in \mathbb{R}_+^* \times \mathbb{R}$ is $\mathfrak{K}_\mathscr{C}(\mathcal{H})_{|S_{q,t}, J}$.
\end{lemma}

\begin{proof}
    Let $(q,t) \in \mathbb{R}_+^*\times \mathbb{R}$ and let $a \in \mathfrak{K}_\mathscr{C}(\mathcal{H})_{|S_{q,t}, J}$. Assume first that $t \neq 0$. Then $\mathfrak{K}_\mathscr{C}(\mathcal{H})_{|S_{q,t}, J} = \mathfrak{K}_\mathscr{C}(\mathcal{H})_{|S_{q,t}}$ and there exists $\tilde a \in \mathfrak{K}_\mathscr{C}(\mathcal{H})_{|S\setminus\tau^{-1}(0)} = \mathfrak{K}_\mathscr{C}(\mathcal{H})_{|S\setminus\tau^{-1}(0), J}$ such that $\tilde a_{|S_{q,t}} = a$. Extending $\tilde a$ by zero on $\tau^{-1}(0)$, we get an element $b \in \mathfrak{K}_\mathscr{C}(\mathcal{H})_J$ whose restriction to $S_{q,t}$ is $a$. Now, assume that $t= 0$. Let $a'$ be an element of $\mathfrak{K}_\mathscr{C}(\mathcal{H})_{|\tau^{-1}(0)^c}$ whose restriction to $S_{q,0}^c$ is $a_{|S_{q,0}^c}$. Because of the isomorphism (\ref{Jequ}), there is exists an element $\tilde a$ of $\mathfrak{K}_\mathscr{C}(\mathcal{H})_{|\tau^{-1}(0), J}$ such that $\tilde a_{|\tau^{-1}(0)^c} = a'$. Moreover, by the injectivity of (\ref{Jequ}) for $S_V = S_{q,0}$, we have $\tilde a_{|S_{q,0}} = a$. If $b \in \mathfrak{K}_\mathscr{C}(\mathcal{H})$ is any extension of $\tilde{a}$, then $b$ is $J$-equivariant and its restriction to $S_{q,0}$ is $a$.

    In each case, there exists $b \in \mathfrak{K}_\mathscr{C}(\mathcal{H})_J$ such that $(\mathcal{L}_*b)_{q,t} = a$, which concludes the proof.
\end{proof}

We now come to the definition of the embeddings $C^*_r(\mathbf{G})_{q,t} \hookrightarrow \mathfrak{K}_\mathscr{C}(\mathcal{H})_{|S_{q,t}}$. We adopt here Notation (\ref{fibers}), even if the continuous field of C*-algebras is not yet defined. The construction of these embeddings will depend on the values of $q$ and $t$ so we treat each case separately.
\vspace{6 pt}

\fbox{Case $q \neq 1$, $t \neq 0$.} Consider $q\in \mathbb{R}_+^* \setminus\{1\}$. If $s \in S_{q,1}$ is a continuous parameter, let $\pi_s$ denote the representation of $C^*_{q,r}(G)$ on $\mathcal{H}_s$ obtained by completion of the $(\mathfrak{g},K)_q$-module $\mathrm{ind}_q(e(s),\Lambda(s))$. We recall that the inner product $(\cdot|\cdot)_{0,q}$ on $\mathcal{O}^{e(s)}(K)$ is $U_q(\mathfrak{g})$-invariant, see \cite{YGE}*{§6}. On the other hand, for every $(\sigma,n) \in \{-1,1\}\times \mathbb{Z}_{>0}$, the inner product $(\cdot|\cdot)_{n,q}$ turns $D^+_q(\sigma,n)$ and $D^-_q(\sigma,n)$ into unitary $(\mathfrak{g},K)_q$-modules (this can be obtained from \cite{DCDTsl2R}*{Proposition 3.14} by explicit computations). If $s \in S_{q,1}\cap \mathrm{Dis}_{\sigma,n}^\pm$ then we denote by $\pi_s$ the representation of $C^*_{q,r}(G)$ on $\mathcal{H}_s$ corresponding to the completion of $D^\pm_q(\sigma,n)$.

Each of these representations of $C^*_{q,r}(G)$ is irreducible except when $s$ is an odd continuous parameter such that $\Lambda(s) \in \{-1,1\}$. In that case, $\pi_s$ is the sum of two irreducible subrepresentations, namely $\mathcal{D}_{q,0,+}$ and $\mathcal{D}_{q,0,-}$. Thus, assembling everything and using that $C^*_{q,r}(G)$ is liminal \cite{YGE}*{Corollary 5.10}, we get a C*-algebra morphism

$$\pi_{q,1} :\quad C^*_{q,r}(G) \to \prod_{s \in S_{q,1}}\mathfrak{K}(\mathcal{H}_s), \quad x \mapsto (\pi_s( x))_{s \in S_{q,1}}.$$
\begin{theorem} \label{piq1}
    The map $\pi_{q,1}$ induces an isomorphism $C^*_{q,r}(G) \to \mathfrak{K}_\mathscr{C}(\mathcal{H})_{|S_{q,1}}$.
\end{theorem}

    \begin{proof}
        First, let us check that the image of $C^*_{q,r}(G)$ through $\pi_{q,1}$ lies in $\mathfrak{K}(\mathcal{H})_{|S_{q,1}}$. By Lemma \ref{generatorsR_q} and because the image of $R_q(\mathfrak{g},K)$ in $C^*_{q,r}(G)$ is dense, this reduces to proving that for every $n \in \mathbb{Z}$, the images of $T^{n}_{q,n}$, $T^{n\pm2,\pm}_{q,n}$ via $\pi_{q,1}$ are continuous on $S_{q,1}$ and vanish at infinity. For any $s \in S_{q,1}$ and $n \in \mathbb{Z}$, we have
        \begin{gather*}
            \pi_s(T^{n}_{q,n})  = (\Lambda(s)+\Lambda(s)^{-1})E_n^n(s),\\
            \pi_s(T^{n\pm 2,\pm}_{q,n})  = (q^{1\pm n} \Lambda(s) - q^{-1\mp n}\Lambda(s)^{-1})E^{n\pm 2}_n(s),
        \end{gather*}
        where for every $m \in \mathbb{Z}$, the operator $E_n^m(s)$ is defined by
        \begin{equation}\label{ntom} E_n^m(s) : \psi \longmapsto \begin{cases}
            \Vert\zeta_{q,n}\Vert_s^{-2} (\zeta_{q,n}| \psi)_s \zeta_{q,m}&\text{if $n,m\in \mathbb{Z}(s)$}\\
            0 &\text{otherwise.}
        \end{cases}
        \end{equation}
        These formulas follow from (\ref{actionT}), they show that $\pi_{q,1}(T^{n}_{q,n})$ and $\pi_{q,1}(T^{n\pm 2,\pm}_{q,n})$ are continuous vector fields for every $n \in \mathbb{Z}$. Moreover, we have $\pi_s(T^{n}_{q,n}) = \pi_s(T^{n\pm 2,\pm}_{q,n}) = 0$ if $n \notin \mathbb{Z}(s)$, which proves that $\pi_{q,1}(T^{n}_{q,n})$ and $\pi_{q,1}(T^{n\pm 2,\pm}_{q,n})$ also vanish at infinity for all $n \in \mathbb{Z}$. This concludes the proof of the inclusion $\pi_{q,1}(C^*_{q,r}(G) )\subset \mathfrak{K}(\mathcal{H})_{|S_{q,1}}$.

        Since $\mathcal{D}_{q,0,+}$ and $\mathcal{D}_{q,0,-}$ are subrepresentations of $\pi_s$ for every odd continuous parameter $s \in S_{q,1}$ such that $\Lambda(s) \in \{-1,1\}$, we actually have the inclusion $\pi_{q,1}(C^*_{q,r}(G)) \subset\mathfrak{K}_\mathscr{C}(\mathcal{H})_{|S_{q,1}}$. Then one can check that the conditions of the generalized Stone-Weierstrass theorem (Theorem \ref{StoneW}) are satisfied: use the description of the spectrum of $C^*_{q,r}(G)$ and $\mathfrak{K}_\mathscr{C}(\mathcal{H})_{|S_{q,1}}$ given by Proposition \ref{Irrepsqt} and Lemma \ref{irrepsK} respectively. We conclude that $\pi_{q,1} : C^*_{q,r}(G) \to \mathfrak{K}_\mathscr{C}(\mathcal{H})_{|S_{q,1}}$ is an isomorphism.
    \end{proof}

Let us now fix any pair $(q,t) \in \mathbb{R}^*_+\times \mathbb{R}$ such that $q\neq 1$ and $t \neq 0$. We denote by $\varphi_{q,t}$ the unique bijection $S_{q,t} \to S_{q^t,1}$ which preserves the parity $\varepsilon$, the map $\Lambda$ and the connected components of $S$ (that is, $\varphi_{q,t}^{-1}(\mathrm{Pri}_\varepsilon) \subset \mathrm{Pri}_\varepsilon$ and $\varphi_{q,t}^{-1}(\mathrm{Dis}_{\sigma,n}^\pm)\subset \mathrm{Dis}_{\sigma,n}^\pm$ for every $\varepsilon,\sigma,n$). The pullback by $\varphi_{q,t}$ defines an isomorphism of C*-algebras $\varphi_{q,t}^* : \mathfrak{K}_\mathscr{C}(\mathcal{H})_{|S_{q^t,1}} \to  \mathfrak{K}_\mathscr{C}(\mathcal{H})_{|S_{q,t}}$. The map $\pi_{q,t} = \varphi_{q,t}^* \circ \pi_{q^t,1}$ is then an isomorphism of C*-algebras $C^*_r(\mathbf{G})_{q,t} \to \mathfrak{K}_\mathscr{C}(\mathcal{H})_{|S_{q,t}}$.
\vspace{6pt}

\fbox{Case $q \neq 1$, $t = 0$.} Let us fix $q \neq 1$. We recall that $C^*_r(\mathbf{G})_{q,0} = C^*(G_0^\star)$. For every continuous parameter $s \in S_{q,0}$, let us write $\pi_s$ for the representation of $ C^*(G_0^\star)$ on $\mathcal{H}_s$ defined by $\mathrm{Ind}(p^\star_{\Lambda(s)}, M,e(s))$, see Section 2.4. If $s \in S_{q,1}$ is a discrete parameter of order $(n,\pm)$, we denote by $\pi_s$ the representation of $ C^*(G_0^\star)$ on $\mathcal{H}_s$ given by the direct sum $\bigoplus_{n \in \mathbb{Z}(s)} \mathrm{Ind}(p^\star_{\Lambda(s)},K,n)$, where $\mathrm{Ind}(p^\star_{\Lambda(s)},K,n)$ is identified with the subspace $\mathbb{C}\zeta_{n}$ of $\mathcal{H}_s$ for every $n \in \mathbb{Z}(s)$. Each of the representations $\pi_s$ is a Hilbert direct sum of irreducible representations of $C^*(G_0^\star)$. Since the latter is liminal, see \cite{Dana}*{Section 7.5}, we can define the following C*-algebra morphism
$$\pi_{q,0} :\quad  C^*(G_0^\star)  \to \prod_{s \in S_{q,0}} \mathfrak{K}(\mathcal{H}_s), \quad
       x \mapsto (\pi_s(x))_{s \in S_{q,0}}.$$

\begin{proposition}\label{piq0} The morphism $\pi_{q,0}$ is injective and its image is
    $\mathfrak{K}_\mathscr{C}(\mathcal{H})_{|S_{q,0}, J}$.
\end{proposition}

\begin{proof} First, let us prove that the image of $\pi_{q,0}$ is contained in the space of continuous vector fields of $\mathfrak{K}_\mathscr{C}(\mathcal{H})_{|S_{q,0}}$ vanishing at infinity. For every $n \in \mathbb{Z}$ let us denote by $e_n$ the unique measure on $K$ such that $\langle e_n, \zeta_m\rangle = \delta_{n,m}$ for all $m\in\mathbb{Z}$.

Let $f \in C(K\backslash U)$, $n \in \mathbb{Z}$ and let us show that the image of $fe_n \in C^*(G_0^\star)$ through $\pi_{q,0}$ lies in $\mathfrak{K}_\mathscr{C}(\mathcal{H})_{|S_{q,0}}$. For every $s \in S_{q,0}$, the operator $\pi_s(fe_n)$ has rank at most $1$, so it belongs to $\mathfrak{K}(\mathcal{H}_s)$. Moreover, if $\Lambda(s)\in \{-1,1\}$, then $f$ acts on $\mathcal{H}_s$ as the scalar $f({p}^\star_{\Lambda(s)})$, where ${p}^\star_\sigma \,(\sigma \in \{-1,1\})$ denote the $K$-fixed points of $K\backslash U$, see Section 2.4. Therefore, in that case, $\pi_s(fe_n)$ is proportional to the orthogonal projection onto $\mathbb{C}\zeta_n$ (or vanishes if $n \notin\mathbb{Z}(s)$), which implies that $\pi_s(fe_n) \in\mathfrak{K}_\mathscr{C}(\mathcal{H})_s$. Thus, the continuity of $\pi_{q,0}(fe_n)$ as a vector field of $\mathfrak{K}_\mathscr{C}(\mathcal{H})_{|S_{q,0}}$ only needs to be checked on the connected components of $S_{q,1}$ that are not reduced to a single point, which are the subsets of even and odd continuous parameters of $S_{q,1}$. The field $\mathcal{H}$ is trivial over each of these connected components.  If $s \in S_{q,1}$ is a continuous parameter of parity $(-1)^n$ then $\pi_s(fe_n)$ is the composition of the multiplication by $k \in K \mapsto f({p}^\star_{\Lambda(s)}k)$ and the orthogonal projection onto $\mathbb{C}\zeta_n$; if the parity of $s$ differs from $(-1)^n$, we have $\pi_s(fe_n) = 0$. Since $s \in S_{q,0} \mapsto {p}^\star_{\Lambda(s)} \in K\backslash U$ is continuous, $\pi_{q,0}(fe_n)$ is indeed a continuous vector field of $\mathfrak{K}_\mathscr{C}(\mathcal{H})_{|S_{q,0}}$. Finally, $\pi_s(fe_n) = 0$ if $n \notin \mathbb{Z}(s)$, hence $\pi_{q,0}(fe_n)$ vanishes at infinity.

The linear span of elements of the form $fe_n$ for $f\in C(K\backslash U)$ and $n \in \mathbb{Z}$ is dense in $C^*(G_0^\star)$. We conclude that $\pi_{q,0}(C^*(G_0^\star)) \subset \mathfrak{K}_\mathscr{C}(\mathcal{H})_{|S_{q,0}}$.

By construction, for every continuous parameter $s \in S_{q,0}$ and every discrete parameter $s' \in S_{q,0}$ satisfying $\Lambda(s) = \Lambda(s')$ and $e(s) = e(s')$, the operator $J(s',s)$ intertwines $\pi_s$ and $\pi_{s'}$, so that $\pi_{q,0}(C^*(G_0^\star)) \subset \mathfrak{K}_\mathscr{C}(\mathcal{H})_{|S_{q,0}, J}$. To show that $\pi_{q,0}$ induces an isomorphism of $C^*(G_0^\star)$ onto $\mathfrak{K}_\mathscr{C}(\mathcal{H})_{|S_{q,0}, J}$, we apply the generalized Stone Weierstrass theorem (Theorem \ref{StoneW}). Being isomorphic to $\mathfrak{K}_\mathscr{C}(\mathcal{H})_{|S_{q,0}^c} = \mathfrak{K}_\mathscr{C}(\mathcal{H}_{|S_{q,0}^c})$ via (\ref{Jequ}), the C*-algebra $\mathfrak{K}_\mathscr{C}(\mathcal{H})_{|S_{q,0}, J}$ is liminal and its spectrum, given by Lemma \ref{irrepsK}, is in bijection with that of $C^*(G_0^\star)$, see Proposition \ref{spectrumGstar}.
\end{proof}

\fbox{Case $q=1, \,t\neq 0$.} Let us fix $t\neq 0$. If $s \in S_{1,t}$ is a continuous parameter, let $\pi_s$ be the representation of $C^*_r(G)=C^*_r(\mathbf{G})_{1,t}$ on $\mathcal{H}_s$ corresponding to the completion of the $(\mathfrak{g},K)$-module $\mathrm{ind}(e(s),\Lambda(s))$. For any discrete parameter $s \in S_{1,t}$ of order $(n,\pm)$, let us denote by $\pi_s$ the representation of $C^*_r(G)$ on $\mathcal{H}_s$ given by the completion of $D^\pm(n)$ (note that the inner product on $\mathcal{H}_s$ is well chosen for that purpose). Applying \cite{ClareHigsonCrisp}*{Theorem 6.8} in the particular case of $\mathrm{SL}(2,\mathbb{R})$, we obtain the following.

\begin{proposition}
    The map
    $$\pi_{1,t}: \quad C^*_r(G) \to \prod_{s \in S_{1,t}} \mathfrak{B}(\mathcal{H}_s), \quad x\mapsto (\pi_s(x))_{s \in S_{1,t}}.$$
    induces a C*-algebra isomorphism $C^*_r(G) \to \mathfrak{K}_\mathscr{C}(\mathcal{H})_{|S_{1,t}}$.
\end{proposition}

\fbox{Case $q=1,\,t= 0$.} For every continuous parameter $s \in S_{1,0}$, let $\pi_s$ be the representation of $C^*(G_0)$ on $\mathcal{H}_s$ defined by $\mathrm{Ind}( p_{\Lambda(s)}, M,e(s))$. For any discrete parameter $s \in S_{q,1}$ of order $(n,\pm)$, we denote by $\pi_s$ the representation of $C^*(G_0)$ on $\mathcal{H}_s$ given by the Hilbert direct sum $\bigoplus_{n \in \mathbb{Z}(s)} \mathrm{Ind}(p_{\Lambda(s)},K,n)$, where $\mathrm{Ind}(p_{\Lambda(s)},K,n)$ is identified with the subspace $\mathbb{C}\zeta_{n}$ of $\mathcal{H}_s$ for every $n \in \mathbb{Z}(s)$. Since $C^*(G_0)$ is liminal and each representation $\pi_s$ is a Hilbert direct sum of irreducible representations of $C^*(G_0)$, we can define the following C*-algebra morphism
$$\pi_{1,0}: \quad C^*(G_0) \to \prod_{s \in S_{1,0}} \mathfrak{K}(\mathcal{H}_s), \quad
       x \to (\pi_s(x))_{s \in S_{1,0}}.$$

\begin{proposition} $\pi_{q,0}$ induces an isomorphism
    $C^*(G_0) \to \mathfrak{K}_\mathscr{C}(\mathcal{H})_{|S_{1,0}, J}$.
\end{proposition}

We omit the proof since it is very similar to that of Proposition \ref{piq0}.

\vspace{6pt}
\fbox{Conclusion.} For each pair $(q,t)\in \mathbb{R}_+^*\times \mathbb{R}$, we defined a C*-algebra embedding
$$\pi_{q,t} : C^*_r(\mathbf{G})_{q,t} \hookrightarrow \mathfrak{K}_\mathscr{C}(\mathcal{H})_{|S_{q,t}, J}.$$
The final step of the construction consists is the following theorem, which follows from Lemma \ref{subfieldJ}.
\begin{theorem} \label{construction}
    There exists a unique continuous field of C*-algebras $C^*_r(\mathbf{G})$ over $\mathbb{R}_+^* \times \mathbb{R}$ whose fibers are given by (\ref{fibers}) and such that $(\pi_{q,t})_{q,t} : C^*_r(\mathbf{G}) \to \mathcal{L}_\ast\mathfrak{K}_\mathscr{C}(\mathcal{H})$ defines a morphism of continuous fields of C*-algebras over $\mathbb{R}_+^* \times \mathbb{R}$.
\end{theorem}

\subsection{Geometric characterizations of the continuous field at $t = 0$ and $q= 1$}

Our construction of $C^*_r(\mathbf{G})$ explicitly involves the representation theory of its fibers. We mainly chose this point of view because it will make the Mackey analogy easier to derive in the setting of $q$-deformations. We now show how to characterize pieces of this continuous field, without any reference to specific representations, using deformations to the normal cone.

If $M$ is any smooth manifold and $V\subset M$ is an embedded submanifold, let us denote by $\mathrm{D}_M(V)$ the associated deformation to the normal cone \cite{DebordSkandalis}*{§1.1}. We recall that the latter is naturally equipped with a canonical submersion $\mathrm{p}_M^V : \mathrm{D}_M(V) \to \mathbb{R}$ whose fiber over every $h \in \mathbb{R}^*$ is $M$, while the fiber over $0$ is the normal bundle $\mathrm{N}_M(V)$ of the inclusion $V \subset M$. As explained in \cite{DebordSkandalis2}, if $V \subset M$ is a Lie groupoid inclusion, then both $\mathrm{N}_M(V)$ and $\mathrm{D}_M(V)$ inherit a natural Lie groupoid structure such that each of the fibers of $\mathrm{p}_M^V$ is a Lie subgroupoid of $\mathrm{D}_M(V)$. Moreover, for every $h \in \mathbb{R}$, the restriction to the fiber $\mathrm{D}_M(V)_h = (\mathrm{p}_M^V)^{-1}(h)$ induces a surjective C*-algebra morphism $$\mathrm{Res}_h  : C^*_r(\mathrm{D}_M(V)) \to C^*_r(\mathrm{D}_M(V)_h).$$
If $\mathrm{N}_M(V)$ is amenable, there is a unique continuous field of C*-algebras over $\mathbb{R}$ whose fiber at each $h \in \mathbb{R}$ is $C^*_r(\mathrm{D}_M(V)_h)$ and whose C*-algebra of continuous vector fields vanishing at infinity is $C^*_r(\mathrm{D}_M(V))$, the evaluation at any $h \in \mathbb{R}$ identifying with $\mathrm{Res}_h$. This is true for the same reasons as \cite{MAhigson}*{Lemma 6.13}.

Let us consider the restriction of $C^*_r(\mathbf{G})$ to $\cpa^{-1}(1)$. An application of \cite{MEmbeddingReal} in the particular case of $\mathrm{SL}(2,\mathbb{R})$ shows that $C^*_r(\mathbf{G})_{|\cpa^{-1}(1)}$ coincides with the classical Mackey deformation field \cite{MAhigson}*{§6.2}. This means that the continuous field of C*-algebras $\tau_*\left(C^*_r(\mathbf{G})_{|\cpa^{-1}(1)}\right)$ coincides with $C^*_r(\mathrm{D}_G(K))$.

A similar statement holds if we instead restrict the field to $\tau^{-1}(0)$. Let us identify $K$ with the group of elements of $G_0^\star$ having both their source and range equal to the base point $b$ of $K\backslash U$. The Lie groupoid $\mathrm{D}_{G_0^\star}(K)$ coincides with the transformation groupoid relative to the fiberwise action of $K$ on $\mathrm{D}_{K\backslash U}(b)$. It is thus amenable and its C*-algebra identifies canonically with the crossed product $C_0(\mathrm{D}_{K\backslash U}(b)) \rtimes K$. Each fiber $\mathrm{D}_{G_0^\star}(K)_h$ is the transformation groupoid relative to the action of $K$ on $\mathrm{D}_{K\backslash U}(b)_h$, hence is amenable. Finally, the restriction morphisms
$$\mathrm{Res}_h : C_0(\mathrm{D}_{K\backslash U}(b)) \rtimes K \longrightarrow C^*(\mathrm{D}_{G_0^\star}(K)_h) = C_0(\mathrm{D}_{K\backslash U}(b)_h) \rtimes K $$
are induced from the restriction maps $C_0(\mathrm{D}_{K\backslash U}(b)) \to C_0(\mathrm{D}_{K\backslash U}(b)_h)$ by crossed product functoriality.

Recall Convention \ref{normalization}, by which the Pontryagin dual $(\mathfrak{g}/\mathfrak{k})\,\hat{}$ of $\mathfrak{g}/\mathfrak{k}$ is identified with $T_b(K\backslash U)$. By this means, the group C*-algebra of $G_0$, which is canonically isomorphic to $C_0((\mathfrak{g}/\mathfrak{k})\,\hat{}\,) \rtimes K$, is from now identified with $$C_0(T_b(K\backslash U)) \rtimes K = C^*(\mathrm{D}_{G_0^\star}(K)_0) = C^*(\mathrm{N}_{G_0^\star}(K)).$$

\begin{proposition}
    Let $\eta : q \in \mathbb{R}_+^* \mapsto 2 \log q \in \mathbb{R}$. The continuous field of C*-algebras over $\mathbb{R}$ defined by $C_r^*(\mathrm{D}_{G_0^\star}(K))$ coincides with $(\eta\circ\cpa)_*\left(C^*_r(\mathbf{G})_{|\tau^{-1}(0)}\right)$.
\end{proposition}

\begin{proof}

    In view of \cite{Dixmier}*{Proposition 10.2.4}, it is enough to check that for any $x$ in $C_0(\mathrm{D}_{K\backslash U}(b)) \rtimes K$, the vector field $\tilde{x} = (\mathrm{Res}_{\eta(q)}(x))_{q >0}$ is an element of $\cpa_*\left(C^*_r(\mathbf{G})_{|\tau^{-1}(0)}\right)$. Recall from the proof of Proposition \ref{piq0} the definition of the family $(e_n)_{n \in \mathbb{Z}}$ of measures on $K$. Let us fix $f \in C_0(\mathrm{D}_{K\backslash U}(b))$, $n \in \mathbb{Z}$, and let us write $x = fe_n \in C_0(\mathrm{D}_{K\backslash U}(b)) \rtimes K$. By construction (Theorem \ref{construction}), the vector field $\tilde x$ is continuous and vanishes at infinity if and only if $x' = (\pi_s(\tilde{x}_{\cpa(s)}))_{s \in \tau^{-1}(0)}$ belongs to $\mathfrak{K}_\mathscr{C}(\mathcal{H})_{|\tau^{-1}(0),J}$, or equivalently, because of  the isomorphism (\ref{Jequ}), if and only if $x'' = (\tilde{x}_{\cpa(s)})_{s \in \tau^{-1}(0)^c}$ belongs to $\mathfrak{K}_\mathscr{C}(\mathcal{H})_{|\tau^{-1}(0)^c}$.

    For any $s \in \tau^{-1}(0) \subset S$, let us write
    $$\mathbf{p}_s = \begin{cases}
        p^\star_{\Lambda(s)} \in \mathrm{D}_{K\backslash U}(b)_{\eta(q)} = K\backslash U &\text{if $\cpa(s) = q \neq 1$}\\
        p_{\Lambda(s)} \in \mathrm{D}_{K\backslash U}(b)_{0} = (\mathfrak{g}/\mathfrak{k})\,\hat{} &\text{if $\cpa(s) = 1$.}
    \end{cases}$$
    For every $s \in \tau^{-1}(0)^c$, the operator $x''_s$ is the composition of the multiplication by the continuous function $k \mapsto f(\mathbf{p}_sk)$ with the operator $E^n_n(s)$ defined by (\ref{ntom}). Now, from the material of Subsections 2.2 and 2.4, one can check that, under the identifications we made, $s \in \tau^{-1}(0)^c \mapsto \mathbf{p}_s \in \mathrm{D}_{K\backslash U}(b)$ is continuous. This show that $x''$ is an element of $\mathfrak{K}_\mathscr{C}(\mathcal{H})_{|\tau^{-1}(0)^c}$, hence that $\tilde{x}$ belongs to $\cpa_*\left(C^*_r(\mathbf{G})_{|\tau^{-1}(0)}\right)$. The fact that $C_0(\mathrm{D}_{K\backslash U}(b)) \rtimes K$ is densely spanned by elements of the same form as $x$ concludes the proof.
\end{proof}

\section{The Mackey analogy for $q$-deformed $\mathrm{SL}(2,\mathbb{R})$}

In this last section, we give a collection of results concerning the $q$-deformations of $\mathrm{SL}(2,\mathbb{R})$ which are natural analogues of the various statements covered by the Mackey analogy \cites{MEmbeddingReal, AfgoustidisAubert, Afgoustidis3, Afgoustidis1}. We show that for fixed $q\neq 1$, the restriction of the continuous field of C*-algebras $C^*_r(\mathbf{G})$ to $\{q\}\times[0,1]$ can be characterized as the mapping cone of a certain embedding $C^*(G_0^\star) \hookrightarrow C^*_{q,r}(G)$. Then we prove that the latter induces a continuous bijection between the spectra of both algebras and an isomorphism in K-theory.

For the rest of the paper, we fix $q\neq 1$ and we write $S_q = \cpa^{-1}(q)$. We also denote by $C^*_{q,r}(\mathbf{G})$ the continuous field of C*-algebras over $\mathbb{R}$ obtained from $C^*_r(\mathbf{G})_{|\{q\}\times\mathbb{R}}$ by pushforward by $\tau$.

\subsection{The Mackey embedding and bijection}
 We start by the construction of the embedding.
 For any $t\in \mathbb{R}$, let us denote by $\gamma_t$ the unique map $S_{q,1} \to S_{q,t}$ which preserves $\Lambda$, $\varepsilon$ and the connected components of $S_q$. Note that the map $(t,s) \in \mathbb{R} \times S_{q,1} \mapsto \gamma_t(s) \in S_q$ is a homeomorphism. Moreover, for any $s \in S_{q}$, let $v_s : \mathcal{H}_s \to \mathcal{H}_{\gamma_{\tau(s)}^{-1}(s)}$ be the isometry defined by
 $$v_s(\zeta_{q^{\tau(s)},n}) = \frac{\Vert\zeta_{q^{\tau(s)},n}\Vert}{\Vert\zeta_{q,n}\Vert} \zeta_{q,n} \qquad (n \in \mathbb{Z}(s)).$$
 Then, we define an isomorphism of C*-algebras $\beta_t : \mathfrak{K}(\mathcal{H})_{|S_{q,t}} \to \mathfrak{K}(\mathcal{H})_{|S_{q,1}}$ for any $t \in \mathbb{R}$ as follows
 $$\beta_t(a)_s = v_{\gamma_t(s)} a_{\gamma_t(s)} v_{\gamma_t(s)}^* \qquad(a \in \mathfrak{K}(\mathcal{H})_{|S_{q,t}}, s \in S_{q,1}).$$
By construction, $\beta = (\beta_t)_{ t \in \mathbb{R}}$ defines a trivialization of the continuous field of C*-algebras $\tau_*(\mathfrak{K}(\mathcal{H})_{|S_{q}})$. For any $t \in \mathbb{R}$, let $\alpha_t : C^*_{q,r}(\mathbf{G})_t \hookrightarrow C^*_{q,r}(G)$ be the unique morphism of C*-algebras such that the following square commutes:
$$\begin{tikzcd}
C^*_{q,r}(\mathbf{G})_t \ar[r, "\pi_{q,t}"]\ar[d,"\alpha_t"]& \mathfrak{K}_\mathscr{C}(\mathcal{H})_{|S_{q,t}} \ar[d, "\beta_t"]\\
C^*_{q,r}(G) \ar[r,"\pi_{q,1}", "\sim"'] & \mathfrak{K}_\mathscr{C}(\mathcal{H})_{|S_{q,1}}
\end{tikzcd}.$$
A reformulation of the above properties with Theorem \ref{construction} leads to the following.

\begin{proposition}
    For every $t \neq 0$, the morphism $\alpha_t$ is bijective, while for $t=0$, it is only injective. Moreover, the collection $(\alpha_t)_{t \in \mathbb{R}}$ is an embedding of the continuous field of C*-algebras $C^*_{q,r}(\mathbf{G})$ into the constant field over $\mathbb{R}$ with fiber $C_{q,r}^*(G)$.
\end{proposition}

In particular, note that for any $q' \in \mathbb{R}_+^* \setminus\{1\}$, the C*-algebras $C_{q,r}^*(G)$ and $C_{q',r}^*(G)$ are isomorphic. Moreover, we have an analogue of the Mackey bijection.

\begin{theorem}
    There exists a unique bijection $\mu$ from the spectrum of $C_{q,r}^*(G)$ to the spectrum of $C^*(G_0^\star)$ such that for any irreducible representation $\omega$ of $C_{q,r}^*(G)$, the representation $\alpha_0^*\omega$ of $C^*(G_0^\star)$ obtained by pullback by $\alpha_0$ contains $\mu(\omega)$ as a subrepresentation. This map $\mu$ is continuous with respect to the Jacobson topologies, although its inverse is not.
\end{theorem}

\begin{proof}
    Assume that $\mu$ is a bijection $\mathrm{Spec}\,C_{q,r}^*(G) \to \mathrm{Spec}\,C^*(G_0^\star)$ between the spectra of both algebras. Let us identify $\mathrm{Spec}\,C_{q,r}^*(G)$ with the list of irreducible unitarizable $(\mathfrak{g},K)_q$-modules given in Proposition \ref{Irrepsqt}. For every $\varepsilon \in \{-1,1\}$ and $\lambda \in \mathbb{U}_+ \setminus \{\pm 1\}$, we have
    \begin{gather*}
        \alpha_0^*[\mathrm{ind}_q(\varepsilon, \lambda)] = \mathrm{Ind}(p^\star_\lambda, M, \varepsilon),
    \end{gather*}
    so that $\mu[\mathrm{ind}_q(\varepsilon, \lambda)] = \mathrm{Ind}(p^\star_\lambda, M, \varepsilon)$. On the other hand, for any $\sigma \in \{-1,1\}$ and $n \in \mathbb{Z}_{\geq  0}$, we have
    \begin{gather*}
        \alpha_0^*[\mathrm{ind}_q(1, \sigma)] = \bigoplus_{m \in \mathbb{Z}^\mathrm{even}} \mathrm{Ind}(p^\star_\sigma, K, m),\\
        \alpha_0^*[D^\pm(\sigma,n)] = \bigoplus_{\substack{m-n \in \mathbb{Z}^\mathrm{odd}\\ \pm m > n }}\mathrm{Ind}(p^\star_\sigma, K, m).
    \end{gather*}
    From this, an easy inductive argument using the bijectivity of $\mu$ implies shows that
    \begin{gather*}
        \mu [\mathrm{ind}_q(1, \sigma)] =\mathrm{Ind}(p^\star_\sigma, K, 0),\qquad
        \mu[D^\pm(\sigma,n)] = \mathrm{Ind}(p^\star_\sigma, K, \pm (n+1)).
    \end{gather*}
    Conversely, one can check that the map thus defined satisfies the conditions of the theorem. Next, the continuity of $\mu$ derives from the following explicit descriptions of the Jacobson topology for $\mathrm{Spec}\,C_{q,r}^*(G)$ and $\mathrm{Spec}\,C^*(G_0^\star)$.
    
    The topology of $\mathrm{Spec}\,C_{q,r}^*(G)$ is given by Corollary \ref{Jacobson} via the isomorphism~$\pi_{q,1}$: the analogues of the discrete series are isolated, the set of analogues of the even principal series form a connected component homeomorphic to $\mathbb{U}_+$, while the irreducible analogues of the odd principal series form an open half-circle, the representations $D^\pm(\sigma,0)$ glued to its ends depending on $\sigma$. Similarly, since $\mathrm{Res}_{S_{q,0}^c} \circ \pi_{q,0} : C^*(G_0^\star) \to \mathfrak{K}_\mathscr{C}(\mathcal{H})_{|S_{q,0}^c}$ is an isomorphism, see (\ref{Jequ}) and Proposition \ref{piq0}, we deduce from Corollary \ref{Jacobson} an explicit description of $\mathrm{Spec}\,C^*(G_0^\star)$.
\end{proof}

\begin{remark}
    One can consider $C^*(K)$ as the subalgebra of $C^*(G_0^\star)$ generated by the projections $e_n$, as $n$ ranges over $\mathbb{Z}$, and it can thus be embedded into $C_{q,r}^*(G)$ through $\alpha_0$. This way, representations of $C^*(G_0^\star)$ or $C_{q,r}^*(G)$ can be restricted to $C^*(K)$. Let us identify the unitary dual of $K = \mathrm{SO}(2)$ with $\mathbb{Z}$ via $n \mapsto \zeta_n$. For any representation $\omega$ of $C^*(G_0^\star)$ or $C_{q,r}^*(G)$ and any $n \in \mathbb{Z}$, let us call $n$ a $K$-type of $\omega$ if the restriction of the latter to a representation of $K$ admits $n$ as a subrepresentation. If such a $K$-type has minimal absolute value among all the $K$-types of $\omega$, we call it minimal. Then one easily check that the Mackey bijection defined above preserves the sets of minimal $K$-types, as in the classical case.
\end{remark}

\subsection{An analogue of the Connes-Kasparov isomorphism} Finally, we show that the C*-algebras $C_{q,r}^*(G)$ and $C^*(G_0^\star)$ share the same K-theory.

Since the field $C^*_{q,r}(\mathbf{G})_{|[0,1]}$ is trivial away from $0$, the evaluation at $\tau =0$ induces an isomorphism in K-theory
$$\mathrm{K}(\mathrm{ev}_{\tau = 0}) : \mathrm{K}_\ast(C^*_{q,r}(\mathbf{G})_{|[0,1]}) \longrightarrow \mathrm{K}_\ast(C^*(G_0^\star)).$$
The question then arises whether the map $\mathrm{K}(\mathrm{ev}_{\tau = 1})\circ\mathrm{K}(\mathrm{ev}_{\tau = 0})^{-1}$ is an isomorphism from $\mathrm{K}_\ast(C^*(G_0^\star))$ to $\mathrm{K}_\ast(C^*_{q,r}(G))$. Indeed, the same procedure at $q=1$ yields the Connes-Kasparov isomorphism $\mathrm{K}_\ast(C^*(G_0)) \to \mathrm{K}_\ast(C^*_r(G))$, see \cite{BaumConnesHigson}*{p.263}.

One can rephrase this conjecture in terms of the embedding $\alpha_0$. Consider the C*-algebra morphism $\sigma : C^*(G_0^\star) \to C^*_{q,r}(\mathbf{G})_{|[0,1]}$ defined by:
$$\sigma(x)_t = \begin{cases}
    \alpha_t^{-1}(\alpha_0(x)) &\text{if $t \neq 0$}\\
    x &\text{otherwise}\end{cases}, \qquad (x \in C^*(G_0^\star), t \in [0,1]).$$
It is a section of $\mathrm{ev}_{\tau = 0}$ and we have $\mathrm{ev}_{\tau = 1}\circ \sigma = \alpha_0$. It follows that the map $\mathrm{K}(\alpha_0)$ induced by $\alpha_0$ in K-theory is precisely equal to $\mathrm{K}(\mathrm{ev}_{\tau = 1})\circ\mathrm{K}(\mathrm{ev}_{\tau = 0})^{-1}$.

\begin{theorem}
    The map $\mathrm{K}(\alpha_0) : \mathrm{K}_\ast(C^*(G_0^\star)) \to \mathrm{K}_\ast(C^*_{q,r}(G))$ is an isomorphism of abelian groups.
\end{theorem}

The proof is the same as that of Higson in \cite{MAhigson}*{§7}, it also appears in \cites{Afgoustidis3, MonkVoigt}.

\begin{proof} We consider the sequence $(W_m)_{m \in \mathbb{N}}$ of subsets of $\mathbb{Z}$ defined by:
$$W_m = \begin{cases}
    \{0\} \quad \text{if $m=1$}\\
    \{-1,1\}\quad \text{if $m=2$}\\
    \{l\} \quad \text{if $m=2l+1$, $l\geq 1$}\\
    \{-l\} \quad \text{if $m=2l+2$, $l\geq 1$}.
\end{cases}$$
This is a list of all the possible sets of minimal $K$-types of irreducible representations of $C^*_{q,r}(G)$ and $C^*(G_0^\star)$. If $x$ is an element of one of these algebras, let us denote by $(x)$ the ideal generated by $x$. We introduce the following sequence of ideals of $C^*_{q,r}(G)$ and $C^*(G_0^\star)$ respectively:
\begin{gather*}
    \mathscr{J}_m = \bigcap_{n \in W_m}(e_{q,n}), \qquad J_m = \bigcap_{n \in W_m}(e_{n}).
\end{gather*}
We write $\mathscr{Q}_m = (\sum_{l\leq m} \mathscr{J}_l) / (\sum_{l <m}\mathscr{J}_l)$ and $Q_m = (\sum_{l\leq m} J_l) / (\sum_{l <m}J_l)$. Since $\alpha_0$ maps $e_n$ to $e_{q,n}$ for all $n \in \mathbb{Z}$, we have $\alpha_0(J_m) \subset \mathscr{J}_m$ for every $m\in \mathbb{N}$. Hence $\alpha_0$ induces a morphism of C*-algebras $\alpha_{0,m} : Q_m \to \mathscr{Q}_m$ for all $m \in \mathbb{N}$.

The image of $C^*(K)$ in $C^*_{q,r}(G)$ and $C^*(G_0^\star)$ generates a dense ideal in both cases. Hence $\sum_{m \in \mathbb{N}} \mathscr{J}_m$ and $\sum_{m \in \mathbb{N}} J_m$ are respectively dense in $C^*_{q,r}(G)$ and $C^*(G_0^\star)$. Because K-theory commutes with inductive limits, the theorem is proven provided that $\alpha_0$ induces an isomorphism $\mathrm{K}_\ast(\sum_{l\leq m} J_l) \to \mathrm{K}_\ast(\sum_{l\leq m} \mathscr{J}_l)$ for all $m \in \mathbb{N}$. By a standard inductive argument using the functoriality of six term exact sequences, this in turn reduces to proving that $\alpha_{0,m}$ induces an isomorphism in K-theory for all $m \in \mathbb{N}$.

Let us fix $m \in \mathbb{N}$. The spectrum of $\mathscr{Q}_m$ consists of the equivalence classes of irreducible representations of $C^*_{q,r}(G)$ whose set of minimal $K$-types is exactly $W_m$. So for any $n \in W_m$, the element $e_{q,n}$ acts on every irreducible representation of $\mathscr{Q}_m$ as a rank-one projection. Using \cite{MAhigson}*{Lemma 6.1}, we infer that $e_{q,n}\mathscr{Q}_m e_{q,n}$ is Morita equivalent to $\mathscr{Q}_m$. By the same argument, the C*-algebra $e_n Q_me_n$ is Morita equivalent to $Q_m$. Now, from the way $\alpha_0$ has been constructed, one can check that $\alpha_{0,m}(e_n Q_me_n) = e_{q,n}\mathscr{Q}_m e_{q,n}$. To sum up, the three bottom arrows of the diagram
$$\begin{tikzcd}
    Q_m \ar[r, "\alpha_{0,m}"]  & \mathscr{Q}_m \\
    e_n Q_me_n \ar[r, "\alpha_{0,m}", "\sim"'] \ar[u, hookrightarrow]& e_{q,n}\mathscr{Q}_m e_{q,n} \ar[u,hookrightarrow]
\end{tikzcd}$$
induce isomorphisms in K-theory. Thus $\mathrm{K}(\alpha_{0,m}) : \mathrm{K}_\ast(Q_m) \to \mathrm{K}_\ast(\mathscr{Q}_m)$ is an isomorphism and the theorem follows.
\end{proof}

\begin{remark}
    One can compute directly the K-theory of $C^*_{q,r}(G)$ by using standard methods. Each of the representations $D^{\pm}(\sigma,n)$ corresponds to a generator of the $\mathrm{K}_0$ group. One the other hand, the analogues of the odd and even principal series are responsible for additional terms $\mathbb{Z}$ and $\mathbb{Z}^3$, respectively, in $\mathrm{K}_0[C^*_{q,r}(G)]$. We thus have
    $$\mathrm{K}_0[C^*_{q,r}(G)] = \mathbb{Z} \oplus \mathbb{Z}^3 \oplus \big(\bigoplus_{\substack{n \in \mathbb{N}\\\sigma= \pm 1}} \mathbb{Z}\,\big).$$
    The group $\mathrm{K}_1[C^*_{q,r}(G)]$ is zero.
\end{remark}

\end{document}